\documentclass[envcountsame,10pt,a4paper]{article}
\usepackage{latexsym,amsmath} 
\usepackage{times}
\usepackage{psfrag}

\def\vec#1{\mathbf{#1}}

\usepackage[dvips]{epsfig} 
\graphicspath{{./Fig_eps/}{./Eps/}} 
\DeclareGraphicsExtensions{.ps,.eps} 
\usepackage{color} 
\usepackage{a4wide}

\newcommand\Walksat{{\tt Walksat}}

\newcommand\sign{\mathrm{sign}}

\newcommand\PHI{\vec\Phi}

\newcommand\cA{\mathcal{A}}

\newcommand\cD{\mathcal{D}} 
\newcommand\cF{\mathcal{F}} 
 
\newcommand\cE{\mathcal{E}} 
 
\newcommand\cN{\mathcal{N}} 
\newcommand\cQ{\mathcal{Q}}

\newcommand\cI{\mathcal{I}}

\newcommand\cY{\mathcal{Y}}

\newcommand\cZ{\mathcal{Z}}

\def\cE{{\cal E}}

\newcommand\eul{\mathrm{e}} 
\newcommand\eps{\varepsilon}

\newcommand\Erw{\mathrm{E}} 
\newcommand\pr{\mathrm{P}}

\newcommand{\vecone}{\vec{1}}

\newcommand{\Bin}{{\rm Bin}}

\newcommand{\bink}[2] {{{#1}\choose {#2}}}

\newcommand\ra{\rightarrow} 

\newcommand\bc[1]{\left({#1}\right)} 
\newcommand\cbc[1]{\left\{{#1}\right\}} 
\newcommand\bcfr[2]{\bc{\frac{#1}{#2}}} 
 
\newcommand\brk[1]{\left\lbrack{#1}\right\rbrack}

\newcommand\abs[1]{\left|{#1}\right|}

\newcommand\RR{\mathbf{R}} 
 
\newcommand{\Whp}{W.h.p.} 
\newcommand{\whp}{w.h.p.} 
\newcommand{\stacksign}[2]{{\stackrel{\mbox{\scriptsize #1}}{#2}}}

\newcommand\Lem{Lemma}
\newcommand\Prop{Proposition}
\newcommand\Thm{Theorem}
\newcommand\Cor{Corollary}
\newcommand\Sec{Section}

\newcommand\algstyle{}

\newtheorem{definition}{Definition}[section]
\newtheorem{example}[definition]{Example}

\newtheorem{theorem}[definition]{Theorem}
\newtheorem{lemma}[definition]{Lemma}
\newtheorem{proposition}[definition]{Proposition}
\newtheorem{corollary}[definition]{Corollary}

\newtheorem{algorithm}[definition]{Algorithm}
\newtheorem{fact}[definition]{Fact}

\newcommand{\qed}{\hfill$\Box$\smallskip}
\newenvironment{proof}{\emph{Proof.}}{}

\begin{document} 

\title{\bf Analyzing Walksat on Random Formulas\thanks{An extended abstract version of this work appeared in the proceedings of ANALCO 2012.}}
 
\author{
Amin Coja-Oghlan\thanks{Supported by EPSRC grant EP/G039070/1.
		Goethe University, Mathematics Institute, 10 Robert Mayer St, Frankfurt 60325, Germany, {\tt acoghlan@math.uni-frankfurt.de}}
\and
Alan Frieze\thanks{
Carnegie Mellon University, Department of Mathematical Sciences, Pittsburgh, PA~15213, USA, {\tt alan@random.math.cmu.edu}}
} 
\date{\today}


\maketitle 

\begin{abstract}
Let $\PHI$ be a uniformly distributed random $k$-SAT formula with $n$ variables and $m$ clauses.
We prove that the \Walksat\ algorithm from
Papadimitriou (FOCS~1991)/Sch\"oning (FOCS 1999)
finds a satisfying assignment of $\PHI$ in polynomial time \whp\
if $m/n\leq \rho\cdot 2^k/k$ for a certain constant $\rho>0$.
This is an improvement by a factor of $\Theta(k)$ over the best previous analysis of \Walksat\ from
	Coja-Oghlan, Feige, Frieze, Krivelevich, Vilenchik (SODA~2009).

\end{abstract}

\section{Introduction}

Let $\PHI=\PHI_k(n,m)$ be a $k$-CNF on $n$ Boolean variables $x_1,\ldots,x_n$ with $m$ clauses
chosen uniformly at random ($k\geq3$).
The interest in random $k$-SAT stems largely from the \emph{experimental} observation
that for certain densities $r$ the random formula $\PHI$ is a challenging algorithmic benchmark~\cite{Cheeseman,MitchellSelmanLevesque}.
However, \emph{analyzing} algorithms on random formulas is notoriously difficult.
Indeed, the current rigorous results for random $k$-SAT mostly deal with algorithms that are extremely
simple both to state and to analyze, or with algorithms that were specifically designed so as to allow for a rigorous analysis.
More precisely, the present analysis techniques are essentially confined to simple  algorithms that 
aim to construct a satisfying assignment by determining the value of one variable at a time \emph{for good},
without any backtracking or reassigning variables at a later time.
By contrast, most `real-life' satisfiability algorithms actually
rely substantially on reassigning variables. 

Maybe the simplest example of a natural 
algorithm that eludes the standard analysis techniques is 
{\tt Walksat}~\cite{Papadimitriou,Schoning}.
Similar local search algorithms are quite successful in practical SAT-solving~\cite{SelmanKautzCohen}.
Starting from the all-true assignment, \Walksat\ tries to find
a satisfying assignment of its input $k$-CNF formula $\Phi=\Phi_1\wedge\cdots\wedge\Phi_m$ as follows.
If the current assignment $\sigma$ is satisfying, then clearly there is nothing to do and the algorithm terminates.
Otherwise, the algorithm picks an index $i$ such that clause $\Phi_i$ is unsatisfied uniformly at random among all such indices.
Clause $\Phi_i$ is a disjunction of $k$ literals $\Phi_{i1}\vee\cdots\vee\Phi_{ik}$.
\Walksat\ picks an index $j\in\cbc{1,\ldots,k}$ uniformly at random and flips the value assigned to the variable 
 underlying the literal $\Phi_{ij}$.
Of course, this ensures that under the new assignment clause $\Phi_i$ is satisfied, but flipping $\Phi_{ij}$
may create new unsatisfied clauses.
If after a certain number $T_{\max}$ of iterations  no satisfying assignment is found, \Walksat\ gives
up and concedes failure.
The pseudocode is shown in Figure~\ref{Fig_Walksat}.
In the worst case, it can be shown that $(2-2/k)^{(1+o(1))n}$ executions of \Walksat\ with independent coins tosses will find a satisfying assignment of a satisfiable
input formula $\Phi$ on
$n$ variables with probability $1-o(1)$,
	for a suitable $T_{\max}=T_{\max}(k)=O(n)$~\cite{Schoning}.

\begin{figure}
\noindent{
\begin{algorithm}\label{Alg_Walksat}\upshape\texttt{Walksat$(\Phi,T_{\max})$}\\\sloppy
\emph{Input:} A $k$-CNF $\Phi=\Phi_1\wedge\cdots\wedge\Phi_m$ over the variables $x_1,\ldots,x_n$
		and a number $T_{\max}\geq0$.\\
\emph{Output:} An assignment $\sigma:V\rightarrow\cbc{0,1}$.
\vspace{-2mm}
\begin{tabbing}
mmm\=mm\=mm\=mm\=mm\=mm\=mm\=mm\=mm\=\kill
{\algstyle 0.}	\> \parbox[t]{30em}{\algstyle
	Initially, let $\sigma(x_i)=1$ for $i=1,\ldots,n$.}\\
{\algstyle 1.}	\> \parbox[t]{30em}{\algstyle
	Repeat the following $T_{\max}$ times (with independent random choices)}\\
{\algstyle 2.}	\> \> \parbox[t]{40em}{\algstyle
		If $\sigma$ is a satisfying assignment, then halt and output $\sigma$.}\\
{\algstyle 3.}	\> \> \parbox[t]{40em}{\algstyle		Otherwise, choose an index $i$ such that clause $\Phi_i$ is unsatisfied under $\sigma$ uniformly at random.}\\
{\algstyle 4.}	\> \> \parbox[t]{38em}{\algstyle
		Suppose that $\Phi_i=\Phi_{i1}\vee\cdots\vee\Phi_{ik}$.\\
		Choose an index $j\in\cbc{1,\ldots,k}$ uniformly at random.\\
			Flip the value of the variable underlying the literal $\Phi_{ij}$ in the assignment $\sigma$.}\\
{\algstyle 5.}	\> \parbox[t]{30em}{\algstyle
	Return `failure'.}
\end{tabbing}
\end{algorithm}}
\caption{The \Walksat\ algorithm.}\label{Fig_Walksat}
\end{figure}

Although \Walksat\ is conceptually very simple, analyzing this algorithm on random formulas is a challenge.
Indeed, \Walksat\ does not follow the naive template of the previously analysed algorithms that assign one variable
at a time for good, because its random choices may (and will) lead \Walksat\ to
flipping quite a few variables several times over.
This causes stochastic dependencies that seem to render the differential equation method, the mainstay of the previous analyses
of random $k$-SAT algorithms, useless.
The goal of the present paper is to present an analysis of \Walksat\ via a different approach
that allows us to deal with the stochastic dependencies.
Our main result is as follows.

\begin{theorem}\label{Thm_pos}
There is a constant $k_0>3$ such that for any $k\geq k_0$
and $$0<m/n\leq\frac1{25}\cdot 2^k/k,$$
\Walksat$(\PHI,\lceil n/k\rceil )$ outputs a satisfying assignment \whp
\end{theorem}

\subsubsection{Related work.}
To put \Thm~\ref{Thm_pos} in perspective, let us compare it with other results on random $k$-SAT algorithms.
The simplest conceivable one is presumably {\tt UnitClause}.
Considering all variables unassigned initially, {\tt UnitClause} sets one variable at a time as follows.
If there is a clause in which $k-1$ variables have been assigned already without satisfying that clause (a `unit clause'),
the algorithm has to assign the $k$th variable so as to satisfy the unit clause.
If there is no unit clause, a currently unassigned variable is chosen randomly and is assigned a random truth value.
As {\tt UnitClause} is extremely simple and does not backtrack, it can be analyzed via the method of
differential equations~\cite{AchDiffEq}.
The result is that {\tt UnitClause} finds a satisfying assignment with a non-vanishing probability
so long as $m/n<(1-o_k(1))\frac{\eul}2\cdot 2^k/k$, where $o_k(1)$ hides a term that tends to $0$ as $k$ gets large~\cite{ChaoFranco2}.
Furthermore, {\tt ShortestClause}, a natural generalization of {\tt UnitClause}, succeeds for
	$m/n<(1-o_k(1))\eul^2/8\cdot 2^k/k$
with \emph{high} probability~\cite{mick}.
Indeed, the algorithm can be modified so as to succeed with high probability even
for $m/n<(1.817-o_k(1))\cdot 2^k/k$ by allowing a \emph{very} limited amount of backtracking~\cite{FrSu}.
Finally, the algorithm {\tt Fix} from~\cite{BetterAlg}, which was specifically designed
for solving random $k$-SAT instances, succeeds up to $m/n<(1-o_k(1))2^k\ln(k)/k$.
By comparison, non-constructive arguments show that 
the threshold for the \emph{existence} of a satisfying assignment
is $(1+o_k(1))\cdot2^k\ln2$~\cite{nae,yuval}.

In summary, \Thm~\ref{Thm_pos} shows that \Walksat\ is broadly competitive with the other known algorithms for random $k$-SAT.
That said, the main point of this paper is not to produce a better algorithmic
bound for random $k$-SAT, but to address the methodological challenge of
analyzing algorithms such
as \Walksat\ that may reassign variables.
This difficult aspect did not occur or was sidestepped in the aforementioned previous analyses~\cite{AchDiffEq,mick,BetterAlg,FrSu}.
Indeed, the lack of techniques for such analyses is arguably one of the most important shortcomings
of the current theory of random discrete structures.

\Thm~\ref{Thm_pos} improves substantially on the previous analyses of \Walksat, at least for general $k$.
The best previous result for this case showed that \whp\ \Walksat\ will find
a satisfying assignment with $T_{\max}=n$ if $m/n<\rho'\cdot2^k/k^2$,
for a certain constant $\rho'>0$~\cite{CFFKV}.
The proof of this result is based on a rather simple observation that allows to sidestep
the analysis of the stochastic dependencies that arise in the execution of \Walksat.
However, it is not difficult to see that this argument is confined to
clause/variable densities $m/n<2^k/k^2$.
\Thm~\ref{Thm_pos} improves this result by a factor of $\Theta(k)$.

Furthermore, the techniques of Alekhnovich and Ben-Sasson~\cite{BenSassonAlke} show that
for any $k$ \Walksat\ will \whp\ find a satisfying assignment within $O(n)$ iterations
if $m/n<r_{k-pure}$, where $r_{k-pure}$ is the `pure literal threshold'.
The analysis in~\cite{BenSassonAlke} depends heavily on the fact that the combinatorial
structure of the hypergraph underlying the random $k$-CNF $\PHI$ is extremely simple
for $m/n<r_{k-pure}$.
Furthermore, because $r_{k-pure}\ra0$ in the limit of large $k$~\cite{Molloy}, this result is quite weak for general $k$.
Yet \cite{BenSassonAlke} remains the best known result for `small' $k$.
For instance, in the case $k=3$ the pure literal bound is $r_{3-pure}\approx1.63$~\cite{Broder}.

Monasson and Semerjian~\cite{Monasson} applied non-rigorous techniques from statistical mechanics
to study the \Walksat\ algorithm on random formulas.
Their work suggests that \Walksat$(\PHI,O(n))$ will find a satisfying assignment \whp\ if
	$m/n<(1-o_k(1))2^k/k$.
\Thm~\ref{Thm_pos} confirms this claim, up to the constant factor $1/25$.

In contrast to the previous `indirect' attempts at analyzing \Walksat\ on random formulas~\cite{BenSassonAlke,CFFKV},
in the present paper we develop a technique for tracing the execution of the algorithm directly.
This allows us to keep track of the arising stochastic dependencies explicitly.
Before we outline our analysis, we need some notation and preliminaries.

\section{Preliminaries}\label{Sec_pre}

We let $\Omega_k(n,m)$ be the set of all $k$-SAT formulas
with variables from $V=\{x_1,\ldots,x_n\}$ that contain exactly $m$ clauses.
To be precise, we consider each formula an ordered $m$-tuple of clauses
and each clause  an ordered $k$-tuple of literals, allowing both
literals to occur repeatedly in one clause and clauses to occur repeatedly in the formula.
Thus, $\abs{\Omega_k(n,m)}=(2n)^{km}$.
Let $\Sigma_k(n,m)$ be the power set of $\Omega_k(n,m)$,
and let $\pr=\pr_k(n,m)$ be the uniform probability measure.
Throughout, we assume that $m=\lceil rn\rceil$
for a fixed number $r>0$, the {\em density}.

As indicated above, we denote a uniformly random element of $\Omega_k(n,m)$ by $\PHI$.
In addition, we use the symbol $\Phi$ to denote specific (i.e., non-random) elements of $\Omega_k(n,m)$.
If $\Phi\in\Omega_k(n,m)$, then $\Phi_i$ denotes the $i$th clause of $\Phi$, and
$\Phi_{ij}$ denotes the $j$th literal of $\Phi_i$.
If $Z\subset\brk m$ is a set of indices, then we let $\Phi_Z=\bigwedge_{i\in Z}\Phi_i$.
If $l\in\cbc{x_1,\bar x_1,\ldots,x_n,\bar x_n}$ is a literal, then we denote its underlying variable by $|l|$.
Furthermore, we define $\sign(l)=-1$ if $l$ is a negative literal, and $\sign(l)=1$ if $l$ is positive.

Recall that a \emph{filtration} is a sequence $(\cF_t)_{0\leq t\leq \tau}$ of $\sigma$-algebras
$\cF_t\subset\Sigma_k(n,m)$ such that $\cF_t\subset\cF_{t+1}$ for all $0\leq t<\tau$.
For a random variable $X:\Omega_k(n,m)\rightarrow\RR$ we let $\Erw\brk{X|\cF_t}$ denote
the \emph{conditional expectation}.
Thus,
	$\Erw\brk{X|\cF_t}:\Omega_k(n,m)\rightarrow\RR$ is a $\cF_t$-measurable
	random variable such that for any $A\in\cF_t$ we have
	$$\sum_{\Phi\in A}\Erw\brk{X|\cF_t}(\Phi)=\sum_{\Phi\in A}X(\Phi).$$
Also remember that $\pr\brk{\cdot|\cF_t}$ assigns a probability measure
	$\pr\brk{\cdot|\cF_t}(\Phi)$ to any $\Phi\in\Omega_k(n,m)$,
namely
	$$\pr\brk{\cdot|\cF_t}(\Phi):A\in\Sigma_k(n,m)\mapsto\Erw\brk{\vecone_A|\cF_t}(\Phi),$$
where $\vecone_A$
 is the indicator of the event $A$.
We need the following well-known bound.

\begin{lemma}\label{Lemma_filt}
Let $(\cF_t)_{0\leq t\leq \tau}$ be a filtration 
and let $(X_t)_{1\leq t\leq \tau}$ be a sequence of non-negative random variables such that each $X_t$ is $\cF_t$-measurable.
Assume that there are numbers $\xi_t\geq0$ such that
	$\Erw\brk{X_t|\cF_{t-1}}\leq\xi_t$ for all $1\leq t\leq\tau$.
Then $\Erw[\prod_{1\leq t\leq \tau}X_t|\cF_0]\leq\prod_{1\leq t\leq \tau}\xi_t$.
\end{lemma}
\begin{proof}
For $1\leq s\leq\tau$ we let $Y_s=\prod_{t=1}^sX_t$.
Let $s>1$.
Since $Y_{s-1}$ is $\cF_{s-1}$-measurable, we obtain
	\begin{eqnarray*}
	\Erw\brk{Y_s|\cF_0}&=&\Erw\brk{Y_{s-1}X_s|\cF_0}
		=\Erw\brk{\Erw\brk{Y_{s-1}X_s|\cF_{s-1}}|\cF_0}=
			\Erw\brk{Y_{s-1}\Erw\brk{X_s|\cF_{s-1}}|\cF_0}
			\leq\xi_s\Erw\brk{Y_{s-1}|\cF_0},
	\end{eqnarray*}
whence the assertion follows by induction.
\qed\end{proof}

\noindent
We also need the following tail bound (``Azuma-Hoeffding'', e.g.~\cite[p.~37]{JLR}).

\begin{lemma}\label{Azuma}
Let $(M_t)_{0\leq t\leq\tau}$ be a super-martingale with respect to a 
filtration $(\cF_t)_{0\leq t\leq \tau}$ such that $M_0=0$.
Suppose that there exist numbers $c_t$ such that $|M_t-M_{t-1}|\leq c_t$ for all $1\leq t\leq\tau$.
Then  for any $\lambda>0$ we have
	$\pr\brk{M_\tau>\lambda}\leq\exp\brk{-\lambda^2/(2\sum_{t=1}^\tau c_t^2)}.$
\end{lemma}

A $k$-CNF $\Phi=\Phi_1\wedge\cdots\wedge\Phi_m$ gives rise to a bipartite graph
whose vertices are the variables $V$ and the clauses $\cbc{\Phi_i:i\in\brk m}$, and in which each clause is adjacent to all the variables that occur in it.
This 
is the \emph{factor graph} of $\Phi$.
For a vertex $v$ of the factor graph we denote by $N(v)=N_\Phi(v)$ the neighborhood of $v$ in the factor graph.
For a set $Z\subset\brk m$ we let $N(\Phi_Z)=\bigcup_{i\in Z}N(\Phi_i)$ be the set of all variables
that occur in the sub-formula $\Phi_Z$.

Let $A,B$ be two disjoint sets of vertices of the factor graph.
Recall that a \emph{$l$-fold matching from $A$ to $B$} is a set $M$ of $A$-$B$-edges 
such that each $a\in A$ is incident with precisely $l$ edges from $M$, while
each $b\in B$ is incident with at most one edge from $M$.
We will make use of the following simple expansion property of the factor graph of random formulas.

\begin{lemma}\label{Lemma_Hall}
There is a constant $k_0>0$ such that for all
$k\geq k_0$ and for $m/n\leq2^k\ln 2$ 
the random formula $\PHI$ has the following property \whp\ 
	\begin{equation}\label{eqHall}
	\parbox[t]{13cm}{For any set $Z\subset\brk m$ of size $\abs Z\leq n/k^2$
	there is a $0.9k$-fold matching from $\PHI_Z$ to $N(\PHI_Z)$.}
	\end{equation}
\end{lemma}
\begin{proof}
We start by proving that \whp\ the random formula $\PHI$ has the following property.
	\begin{equation}\label{eqLemmaHall}
	\parbox[t]{13cm}{For any set $U$ of $\leq n/k$ variables we have
	$\abs{\cbc{i\in\brk m:N(\PHI_i)\subset U}}\leq1.1|U|/k.$}
	\end{equation}
To prove~(\ref{eqLemmaHall}) we use a `first moment' argument.
For set $U\subset V$ we let $X_U=1$ if $\abs{\cbc{i\in\brk m:N(\PHI_i)\subset U}}>1.1|U|/k$, and we set $X_U=0$ otherwise.
Then
	$$\Erw\brk{X_U}=\pr\brk{X_U=1}\leq\bink{m}{1.1|U|/k}(|U|/n)^{1.1|U|}.$$
Furthermore, for any $1\leq u\leq n/k$ we let $X_u=\sum_{U\subset V:|U|=u}X_U$.
Assuming that $k\geq k_0$ is sufficiently large, we obtain
	\begin{eqnarray*}
	\Erw\brk{X_u}&\leq&\sum_{U\subset V:|U|=u}\Erw\brk{X_U}
			\leq\bink{n}{u}\bink{m}{1.1u/k}\bcfr un^{1.1 u}\\
		&\leq&
		\brk{\frac{\eul n}u\cdot\brk{\bcfr{\eul m}{1.1u/k}^{1/k}\cdot\frac un }^{1.1}}^{u}
			\leq\brk{\frac{\eul n}u\brk{\bc{\frac{\eul2^kk\ln2}{1.1}\cdot\frac nu}^{1/k}\cdot\frac un}^{1.1}}^{u}\\
			&\leq&\brk{\eul\bcfr{u}n^{0.1-1/k}\bcfr{\eul2^kk\ln2}{1.1}^{1.1/k}}^{u}
				\leq\brk{\eul^2\bcfr{u}n^{0.09}}^u.
	\end{eqnarray*}
Summing the last expression over $1\leq u\leq n/k$ and assuming that $k\geq k_0$ is large enough,
we see that
	\begin{eqnarray*}
	\Erw\sum_{1\leq u\leq n/k}X_u&\leq&\sum_{1\leq u\leq\ln^2 n}\brk{\eul^2\bcfr{u}n^{0.09}}^u+\sum_{\ln^2n<u\leq n/k}\brk{\eul^2k^{-0.09}}^u\\
		&\leq&\ln^2n\cdot\eul^2(\ln^2n/n)^{0.09}+\frac nk\cdot\brk{\eul^2k^{-0.09}}^{\ln^2n}=o(1).
	\end{eqnarray*}
Thus, $\sum_{1\leq u\leq n/k}X_u=0$ \whp\ by Markov's inequality.
Hence, (\ref{eqLemmaHall}) holds true \whp

Now, assume that
$\PHI$ satisfies~(\ref{eqLemmaHall}).
Let $Z\subset\brk m$ be a set of size $|Z|\leq n/k^2$.
Let $Y\subset Z$ and let $U=N(\PHI_Y)$.
Then $|U|\leq n/k$, and $N(\PHI_i)\subset U$ for any $i\in Y$.
Therefore, (\ref{eqLemmaHall}) implies that $|Y|\leq1.1|U|/k$,
i.e., $|U|\geq\frac{k}{1.1}|Y|\geq0.9k|Y|$.
Hence, the assertion follows from the marriage theorem.
\qed\end{proof}

\noindent
The following lemma states a second expansion-type property.

\begin{lemma}\label{Lemma_core}
There exists a constant $k_0>0$ 
such that for all $k\geq k_0$ and for any $\eps>0,\lambda>4$ satisfying
$\eps\leq k^{-3}$ and
$\eps^\lambda\leq\frac1\eul(2\eul)^{-4k}$
the random formula $\PHI$ with $m/n\leq2^k\ln2$ has the following property \whp\ 
\begin{equation}\label{eqcore}
\parbox[c]{12cm}{
Let $Z\subset\brk m$ be any set of size $\abs Z\leq\eps n$.
If  
$i_1,\ldots,i_l\in\brk m\setminus Z$ is a sequence of pairwise distinct indices such that
	$$|N(\PHI_{i_s})\cap N(\PHI_{Z\cup\cbc{i_j:1\leq j<s}})|\geq\lambda\quad\mbox{ for all }1\leq s\leq l,$$
then $l\leq\eps n$.
}
\end{equation}
\end{lemma}
\begin{proof}
It is clearly sufficient to prove that the desired property holds \whp\ for all sets $Z$ of size \emph{precisely} $|Z|=\eps n$.
Assume that there is a set $Z$ and a sequence $\vec i=(i_1,\ldots,i_l)$ of pairwise distinct indices in $\brk m\setminus Z$ of length $l=\eps n$ such that
$|N(\PHI_{i_s})\cap N(\PHI_{\cI\cup\cbc{i_j:1\leq j<s}})|\geq\lambda\mbox{ for all }1\leq s\leq l$.
Then the sets $Y=\bigcup_{j=1}^l N(\PHI_{i_j})\setminus N(\PHI_Z)\subset V$ and $Z$ have the following properties.
\begin{enumerate}
\item[a.] $|Y|\leq\eps(k-\lambda)n$.
\item[b.] There is a set $I\subset\brk m\setminus Z$ of size $|I|=\eps n$ such that
			$N(\PHI_i)\subset N(\PHI_Z)\cup Y$ for all $i\in I$.
\end{enumerate}
Property a.\ holds because each clause $\PHI_{i_j}$ adds no more than $k-\lambda$ `new' variables to $Y$, and b.\
is true for the set $I=\cbc{i_j:1\leq j\leq l}$.

To prove that \whp\ there do not exist $Z$ and $\vec i$ of length $l=\eps n$ as above, we are going to show by a first moment
argument that \whp\ the random formula $\PHI$ does not feature sets $Y,Z$ that satisfy a.\ and b.
More precisely, for sets $Z\subset\brk m$ of size $|Z|=\eps n$, $Y\subset V$ of size $|Y|=\eps(k-\lambda)n$,
and $I\subset\brk m\setminus Z$ of size $|I|=\eps n$ we let $\cE(Z,Y,I)$ be the event that $N(\PHI_i)\subset N(\PHI_Z)\cup Y$ for all $i\in I$.
Then for any fixed $Z,Y,I$ we have
	\begin{eqnarray*}
	\pr\brk{\cE(Z,Y,I)}&\leq&\bcfr{k|Z|+|Y|}{n}^{k|I|}\leq(\eps(2k-\lambda))^{k\eps n},
	\end{eqnarray*}
because each of the $k|I|$ variable occurrences in the clauses $\PHI_I$ is uniformly distributed over $V$.
Hence, by the union bound, for large enough $k$
	\begin{eqnarray}
	\pr\brk{\exists Z,Y,I:\cE(Z,Y,I)}&\leq&\sum_{Z,Y,I}\pr\brk{\cE(Z,Y,I)}
			\leq\bink{m}{\eps n}^2\bink{n}{\eps n(k-\lambda)}(\eps(2k-\lambda))^{k\eps n}\nonumber\\
		&\leq&\brk{\bcfr{\eul m}{\eps n}^2\bcfr{\eul}{\eps(k-\lambda)}^{k-\lambda}(\eps(2k-\lambda))^k}^{\eps n}\nonumber\\
		&\leq&\brk{\bcfr{\eul2^k}{\eps}^2\bcfr{\eul(2k-\lambda)}{k-\lambda}^{k-\lambda}(2k\eps)^\lambda}^{\eps n}\nonumber\\
		&\leq&\brk{\bcfr{\eul2^k}{\eps}^2\exp\bc{2k}(2k\eps)^\lambda}^{\eps n}
				\leq\brk{\bc{2\eul}^{2k}\eps^{\lambda/2}}^{\eps n}, 
		\label{eqEZYI}
	\end{eqnarray}
where the last inequality follows from our assumption that $\eps\leq k^{-3}$ with $k\geq k_0$ sufficiently large.
Due to our assumption that $\eps^\lambda\leq\frac1\eul\eul(2\eul)^{-4k}$, (\ref{eqEZYI}) yields
	$\pr\brk{\exists Z,Y,I:\cE(Z,Y,I)}\leq\exp(-\eps n)=o(1)$,
whence the assertion follows.
\qed\end{proof}

Finally, it will be convenient to assume in our proof of \Thm~\ref{Thm_pos} that the formula density $r=m/n$ is `not too small'
and that the clause length $k$ is sufficiently large.
These assumptions are justified as the case of small $k$ or very small $r$ is already covered by~\cite{CFFKV}.

\begin{theorem}[\cite{CFFKV}]\label{Thm_CFFKV}
There is a constant $k_0>3$ such that for all $k\geq k_0$ and all $r\leq \frac16\cdot 2^k/k^2$ \whp\
\Walksat$(\PHI,n)$ will find a satisfying assignment.
\end{theorem}

\section{Outline of the analysis}\label{Sec_Outline}

{\em Throughout this section we assume that $k\geq k_0$ for some large enough constant $k_0>0$,
	and that $r=m/n\sim\rho\cdot2^k/k$ with $k^{-2}\leq \rho<\rho_0=1/25$. 
	We can make these assumptions as otherwise the assertion of \Thm~\ref{Thm_pos} already follows from \Thm~\ref{Thm_CFFKV}.
	Furthermore, let
		\begin{equation}\label{eqlambdaeps}
		\lambda=\sqrt k\mbox{ and }\eps=\exp(-k^{2/3}).\end{equation}}%
The standard approach to analyzing an algorithm on random $k$-SAT formulas is the \emph{method of deferred decisions},
which often reduces the analysis to the study of a system of ordinary differential equations that capture the dynamics of the algorithm~\cite{AchDiffEq}.
Roughly speaking, the method of deferred decisions applies where the state of the algorithm after a given number of steps
can be described by a simple probability distribution, depending only on a very few parameters determined by the past decisions of the algorithm.
This is typically so in the case of simple backtrack-free algorithms such as {\tt UnitClause}.

However, in the case of \Walksat, this approach does not apply because
the algorithm is bound to flip many variables more than once.
This entails that the algorithms' future steps depend on past events in a more complicated way than the method
of deferred decisions can accommodate.
Hence, our approach will be to use the method of deferred decisions to trace the effect of flipping
a variable \emph{for the first time}.
But we will need additional arguments to deal with the dependencies that arise out of flipping
the same variable several times.

To get started, let us investigate the effect of the \emph{first} flip that \Walksat\ performs.
Let $\sigma=\vecone$ be the assignment that sets every variable to true.
Clearly, a clause $\PHI_i$ is unsatisfied under $\sigma$ iff it consists of negative literals only.
As $\PHI$ consists of $m$ uniformly random and independent clauses, the number of unsatisfied
clauses has a binomial distribution $\Bin(m,2^{-k})$, and thus there will be $(1+o(1))2^{-k}m\sim \rho n/k$ all-negative clauses \whp\
To perform its first flip, \Walksat\ chooses an index $i\in\brk m$ such that $\PHI_i$ is all-negative uniformly at random,
then chooses a literal index $j\in\brk k$ uniformly, and sets $\sigma(|\PHI_{ij}|)$ to false,
thereby satisfying clause $\PHI_i$.

But, of course, flipping $|\PHI_{ij}|$ may well generate new unsatisfied clauses.
We need to study their number.
As $\PHI_i$ is just a uniformly random all-negative clause, the random variable $|\PHI_{ij}|$ is uniformly
distributed over the set of all $n$ variables, and thus we may assume without loss that $|\PHI_{ij}|=x_1$.
Furthermore, if a clause $\PHI_l$ becomes unsatisfied because variable $x_1$ got flipped, then $x_1$ must
have been the only variable that appears positively in $\PHI_l$.
Now, the number of clauses whose only positive literal is $x_1$ has distribution $\Bin(m,k/(n2^k)+O(1/n^2))$.
Indeed, the probability that a random clause has precisely one positive literal is $k/2^k$,
	and the probability that this positive literal happens to be $x_1$ is $1/n$;
	the $O(1/n^2)$ accounts for the number of clauses in which variable $x_1$ occurs more than once.
Hence, the \emph{expected} number of newly created unsatisfied clauses equals
	$(1+o(1))\frac{km}{2^kn}\sim\rho.$

In summary, as we are assuming that $\rho\leq\rho_0=1/25<1$, the \emph{expected change} in the number of unsatisfied clauses as a result
of the first flip is bounded from above by
	$$\rho-1+o(1)<0.$$
(The precise value is even smaller because $x_1$ may occur
in further all-negative clauses.)
Thus, we expect that the first flip will indeed reduce the number of unsatisfied clauses.
Of course, this simple calculation does not extend to the further steps of \Walksat\
because knowing the outcome of the first flip renders the
various above statements about clauses/literals being uniformly distributed invalid.

To analyze the further flips, we will describe \Walksat\ as a stochastic process.
Our time parameter will be the number of iterations of the main loop
	(Steps~2--4 in Figure~\ref{Fig_Walksat}),
i.e., the number of flips performed.
To represent the conditioning of the random input formula imposed
up to time $t$, we will define a sequence of random maps $(\pi_t)_{t\geq0}$.
These maps reflect for each pair $(i,j)\in\brk m\times\brk k$ the conditional distribution of the literals $\PHI_{ij}$,
	given the information that \Walksat\ has revealed after performing the first $t$ flips.
More precisely, the value of $\pi_t(i,j)$ will either be just the \emph{sign} of the literal $\PHI_{ij}$, or the
actual literal $\PHI_{ij}$ itself.
In the initial map $\pi_0$, we have
$\pi_0(i,j)=\sign(\PHI_{ij})$ for all $(i,j)\in\brk m\times\brk k$.

At times $t\geq1$ the map $\pi_t$ will feature the occurrences of all variables that have been flipped thus far.
That is, for any pair $(i,j)$ such that \Walksat\ has flipped the variable $|\PHI_{ij}|$ at least
once by time $t$, we let $\pi_t(i,j)=\PHI_{ij}$. 
This information will be necessary for us to investigate the effect of flipping the same
variable more than once.

In addition, we need to pay particular attention to clauses that contain many variables that have
been flipped at least once.
The reason is that these clauses have `too little randomness' left for a direct analysis, and thus we will need to study them separately.
More precisely, in our map $\pi_t$ we will fully reveal all clauses $\PHI_i$  
in which at least 
	\begin{equation}\label{eqk1}
	k_1=0.57\,k
	\end{equation}
literals $\PHI_{ij}$ have been flipped at least once.
Furthermore, we will also recursively reveal all clauses that contain at least $\lambda$ variables
from clauses that were fully revealed before.
This recursive process ensures that we can separate the analysis of clauses that are `heavily conditioned' by
the past steps of \Walksat\ from the bulk of the formula.

Throughout this process that mirrors the execution of \Walksat, all variables whose occurrences have been revealed will be labeled
either with an asterisk or with a zero.
Those variables that got revealed because they occur either in a `heavily conditioned' clause or in another
clause that got revealed by the recursive process described in the previous paragraph will be labeled $0$.
All other variables 
that have been flipped by \Walksat\ at least once are labeled $*$.
We will let $\cA_t$ denote the set of all variables labeled $*$, and $\cN_t$ the set of all variables labeled $0$.

Let us now define the maps $\pi_t$ and the sets $\cA_t,\cN_t$ formally.
Each $\pi_t$ is a map $\brk m\times\brk k\ra\cbc{-1,1}\cup L$,
with $L=\cbc{x_1,\bar x_1,\ldots,x_n,\bar x_n}$ the set of literals.
As mentioned above, we let $\pi_0(i,j)=\sign(\PHI_{ij})$ for all $(i,j)\in\brk m\times\brk k$.
Additionally, let $\cA_0=\cN_0=\cZ_0=\emptyset$, and let $\sigma_0:V\ra\cbc{0,1},\, x\mapsto1$ be the all-true assignment.
For a set $S\subset V$ we call a clause $\PHI_i$ \emph{$S$-negative} if for all $j\in\brk k$ with $\sign(\PHI_{ij})=1$
we have $\PHI_{ij}\in S$.
(In other words, $\PHI_i$ is $S$-negative if all of its positive literals lie in $S$.)
For $t\geq1$, we define the maps $\pi_t$ along with the sets $\cA_t,\cN_t,\cZ_t$ inductively via the process shown in Figure~\ref{Fig_PI}.
Intuitively, the set $\cZ_t$ contains the clauses that are `heavily conditioned' at time $t$, and $\cN_t$ is
the set of variables that occur in such clauses.
Moreover, $\cA_t$ is the set of all variables that have been flipped at least once by time $t$ except
the ones that belong to $\cN_t$.

\begin{figure}\normalsize
\begin{tabular}{ll}
{\bf PI0.}&\parbox[t]{14cm}{If the assignment $\sigma_{t-1}$ satisfies $\PHI$, then the process terminates.}\\
{\bf PI1.}&\parbox[t]{14cm}{Otherwise, choose an index $i_t$ such that $\PHI_{i_t}$ is unsatisfied under $\sigma_{t-1}$
		uniformly at random from the set of all such indices.
		In addition, choose $j_t\in\brk k$ uniformly at random.
		Define $\sigma_t:V\ra\cbc{0,1}$
		by letting $\sigma_t(|\PHI_{i_tj_t}|)=1-\sigma_{t-1}(|\PHI_{i_tj_t}|)$ and $\sigma_t(x)=\sigma_{t-1}(x)$ for all $x\neq|\PHI_{i_tj_t}|$.}\\
{\bf PI2.}&\parbox[t]{14cm}{
		Initially, let $\cZ_t=\cZ_{t-1}$ and $\cN_t=\cN_{t-1}$.\\
		While there is an index $i\in\brk m\setminus\cZ_t$ such that
				$\PHI_i$ is $(\cA_{t-1}\cup\cN_t\cup\cbc{|\PHI_{i_tj_t}|})$-negative and either
				\begin{itemize}
				\item there are at least $k_1$ indices $j\in\brk k$ with $|\PHI_{ij}|\in\cA_{t-1}\cup\cbc{|\PHI_{i_tj_t}|}$, or
				\item there are more than $\lambda$ indices $j\in\brk k$	with $|\PHI_{ij}|\in\cN_t$,
				\end{itemize}
			add the least such index $i_{\min}$ to $\cZ_t$ and add the variables
				$\cbc{|\PHI_{i_{\min} j}|:j\in\brk k}$ to $\cN_t$.
	}\\
{\bf PI3.}&\parbox[t]{14cm}{Let $\cA_t=(\cA_{t-1}\cup\cbc{|\PHI_{i_tj_t}|})\setminus\cN_t$.\\
		Define the map $\pi_t:\brk m\times\brk k\ra\cbc{-1,1}\cup L$ by letting
			$$\pi_t(i,j)=\left\{
					\begin{array}{cl}
					\PHI_{ij}&\mbox{ if }|\PHI_{ij}|\in\cA_t\cup\cN_t,\\
					\sign(\PHI_{ij})&\mbox{ otherwise}.
					\end{array}
					\right.$$}
\end{tabular}
\caption{the construction of the maps $\pi_t$}\label{Fig_PI}
\end{figure}

Let $T$ be the stopping time of this process, i.e., the minimum $t$ such that $\sigma_t$ satisfies $\PHI$ (or $\infty$ if there is no such $t$).
For $t>T$, we define $\pi_t=\pi_T$, $\sigma_t=\sigma_T$, $\cA_t=\cA_T$, $\cN_t=\cN_T$, and $\cZ_t=\cZ_T$.

Steps~{\bf PI0}--{\bf PI1} mirror the main loop of the \Walksat\ algorithm;
	in particular, the stopping time $T$ equals the total number of iterations of the main loop of \Walksat\ before
	a satisfying assignment is found.
The purpose of the remaining steps is to `update' the sets $\cA_t$ and $\cZ_t$ and the map $\pi_t$ as described above.
Before we continue, it may be useful to illustrate the construction of the maps $\pi_t$ with an example.

\begin{example}
Let us go through the example of a 5-SAT formula with $6$ clauses on $10$ variables.
For the sake of this example, we will work with $k_1=2$ and $\lambda=2$.
	(Recall that in our proof we actually assume that $k\geq k_0$ is large enough, $k_1$ is as in~(\ref{eqk1}) and
		$\lambda=\sqrt k$.)
We will represent the maps $\pi_t$ by tables whose columns correspond to the clauses $\PHI_i$.
Thus, the $j$th entry in column $i$ represents the value $\pi_t(i,j)$.
To improve readability, we just write $+$ and $-$ instead of $\pm1$.
Suppose that the initial map $\pi_0$, containing the signs of all literals, reads
	$$
	\textcolor{black}{\pi_0=}
	\begin{array}{cccccc}
	-&-&-&+&+&+\\
	-&+&-&+&-&+\\
	-&-&-&-&-&+\\
	-&-&-&-&+&+\\	
	-&-&-&-&+&+
	\end{array}
	$$
The initial assignment $\sigma_0$ is the all-true assignment, and $\cA_0=\cN_0=\cZ_0=\emptyset$.
Throughout, we will mark the variables in $\cA_t$ by an asterisk $*$ and the variables in $\cN_t$ by a $0$.

Being all-negative, clauses $\PHI_1$ and $\PHI_3$ are unsatisfied under $\sigma_0$.
Therefore, at time $t=1$ step {\bf PI1} chooses $i_1\in\cbc{1,3}$ randomly;
	say, the outcome is $i_1=1$.
In addition, {\bf PI1} chooses $j_1\in\brk k=\cbc{1,2,3,4,5}$ uniformly at random.
Suppose the result is $j_1=5$.
To carry on, we need to reveal the variable $|\PHI_{15}|$.
Thus far, the process has not imposed any conditioning on $|\PHI_{15}|$, and therefore
this variable is uniformly distributed over the set of all our $n=10$ variables.
Assume that indeed $|\PHI_{15}|=x_1$.
Then {\bf PI1} sets $\sigma_1(x_1)=0$ and $\sigma_1(x)=1$ for all $x\neq x_1$.

To implement {\bf PI2} we need to reveal all occurrences of $x_1$ in our random formula.
As there is no previous conditioning on any of variables $|\PHI_{ij}|$ with $(i,j)\neq(1,5)$,
these variables remain independently uniformly distributed over the set of all variables,
and thus the events $\cbc{|\PHI_{ij}|=x_1}$ occur independently with probability $1/n$.
Suppose that $x_1$ occurs at the following positions:
	$$
	\textcolor{white}{\pi_0=}
	\begin{array}{cccccc}
	-&-&-&{\bf x_1}&+&+\\
	-&{\bf x_1}&-&+&-&+\\
	-&-&-&-&{\bf \bar x_1}&+\\
	-&-&-&-&+&+\\	
	{\bf \bar x_1}&-&-&-&+&{\bf x_1}
	\end{array}
	$$
Then there is no clause with at least $k_1$ occurrences of a variable from $\cA_0\cup\cN_0\cup\cbc{x_1}=\cbc{x_1}$,
and thus step {\bf PI2} is void.
Hence, at the end of the first iteration we have $\cA_1=\cbc{x_1}$, $\cN_1=\cZ_1=\emptyset$, and
	$$
	\textcolor{black}{\pi_1=}
	\begin{array}{cccccc}
	-&-&-&x_1^*&+&+\\
	-&x_1^*&-&+&-&+\\
	-&-&-&-&\bar x_1^*&+\\
	-&-&-&-&+&+\\	
	\bar x_1^*&-&-&-&+&x_1^*
	\end{array}
	$$

At time $t=2$ there are two unsatisfied clauses: $\PHI_2$, whose only positive literal got flipped to false, and $\PHI_3$, which was unsatisfied initially.
Step {\bf PI1} chooses one of them randomly, say $i_2=2$, and also chooses a random position $j_2\in\brk k$, say $j_2=2$.
As we already know from the first step, the literal in this position is $\PHI_{22}=\pi_1(2,2)=x_1$.
In effect, the second iteration reverses the flip made in the first one and thus $\sigma_2$ is the all-true assignment.
Since we have revealed all the occurrences of $x_1$  already, step {\bf PI2} is void and $\pi_2=\pi_1$,
$\cA_2=\cbc{x_1}$, and $\cN_2=\cZ_2=\emptyset$.

At the start of the third iteration the unsatisfied clauses are $\PHI_1,\PHI_3$.
Suppose {\bf PI1} chooses $i_3=1$ and $j_3=1$.
Then we need to reveal the variable $|\PHI_{11}|$.
At this point, the only conditioning imposed on this variable is that it is different from $x_1$,
because all occurrences of $x_1$ have been revealed already.
Thus, $|\PHI_{11}|$ is uniformly distributed over $x_2,\ldots,x_{10}$.
Suppose that $|\PHI_{11}|=x_2$.
Then $\sigma_3(x_2)=0$ and $\sigma_3(x)=1$ for all $x\neq x_2$.
To reveal the occurrences of $x_2$ all over the formula, note that by the same argument
we applied to $|\PHI_{11}|$ all spots marked $\pm$ in
$\pi_2$ hide variables that are uniformly distributed over $x_2,\ldots,x_{10}$.
Let us assume that $x_2$ occurs in the following positions.
	$$
	\textcolor{white}{\pi_1=}
	\begin{array}{cccccc}
	{\bf \bar x_2}&-&-&x_1^*&+&+\\
	-&x_1^*&-&+&-&+\\
	-&-&{\bf \bar x_2}&-&\bar x_1^*&+\\
	-&-&-&-&+&{\bf x_2}\\	
	\bar x_1^*&-&-&-&+&x_1^*
	\end{array}
	$$

As clause $\PHI_1$ is $\cA_2\cup\cN_2\cup\cbc{x_2}=\cbc{x_1,x_2}$-negative and contains $k_1=2$ occurrences of variables from
	$\cA_2\cup\cbc{x_2}=\cbc{x_1,x_2}$, {\bf PI2} sets $\cZ_3=\cbc{1}$, reveals
the remaining three variables in $\PHI_1$, and adds all variables that occur in $\PHI_1$ to $\cN_3$.
Suppose that the remaining variables in $\PHI_1$ are $|\PHI_{12}|=x_3$, $|\PHI_{13}|=x_4$, $|\PHI_{13}|=x_5$.
Then $\cN_3=\cbc{x_1,x_2,x_3,x_4,x_5}$; in particular, $x_1,x_2$ are now labeled $0$.
The new $0$ label `overwrites' the $*$ because {\bf PI3} ensures that $\cA_3=(\cA_2\cup\cbc{x_2})\setminus\cN_3=\emptyset$.
In order to carry out {\bf PI2}, we need to reveal all occurrences of variables from $\cN_3$.
Suppose this yields
	$$
	\textcolor{white}{\pi_1=}
	\begin{array}{cccccc}
	\bar x_2^0&-&-&x_1^0&+&+\\
	{\bf \bar x_3^0}&x_1^0&-&{\bf x_5^0}&-&+\\
	{\bf \bar x_4^0}&-&-&-&\bar x_1^0&+\\
	{\bf \bar x_5^0}&-&-&{\bf \bar x_4^0}&{\bf x_3^0}&x_2^0\\	
	\bar x_1^0&-&{\bf \bar x_5^0}&{\bf \bar x_3^0}&{\bf x_4^0}&x_1^0
	\end{array}
	$$
Then clause $\PHI_4$ has become $\cA_2\cup\cN_3\cup\cbc{x_2}=\cbc{x_1,\ldots,x_5}$-negative (as there is no $+$-sign left in column four),
and thus {\bf PI2} sets $\cZ_3=\cbc{1,4}$.
To proceed, we need to reveal the remaining $-$-sign of $\PHI_4$, add the underlying variable
to $\cN_3$, and reveal all of its occurrences.
Suppose that this yields
	$$
	\textcolor{white}{\pi_1=}
	\begin{array}{cccccc}
	\bar x_2^0&-&-&x_1^0&+&+\\
	\bar x_3^0&x_1^0&-&x_5^0&-&+\\
	\bar x_4^0&-&-&{\bf \bar x_6^0}&\bar x_1^0&{\bf x_6^0}\\
	\bar x_5^0&-&-&\bar x_4^0&x_3^0&x_2^0\\	
	\bar x_1^0&-&\bar x_5^0&\bar x_3^0&x_4^0&x_1^0
	\end{array}
	$$
At this point {\bf PI2} stops, because clauses $\PHI_5,\PHI_6$ have $+$-signs left
and clauses $\PHI_2,\PHI_3$ contain only one variable labeled $0$.
Thus, at the end of the third iteration we have $\cA_3=\emptyset$, $\cN_3=\cbc{x_1,\ldots,x_6}$,$\cZ_3=\cbc{1,4}$, and
	$$
	\textcolor{black}{\pi_3=}
	\begin{array}{cccccc}
	\bar x_2^0&-&-&x_1^0&+&+\\
	\bar x_3^0&x_1^0&-&x_5^0&-&+\\
	\bar x_4^0&-&-&\bar x_6^0&\bar x_1^0&x_6^0\\
	\bar x_5^0&-&-&\bar x_4^0&x_3^0&x_2^0\\	
	\bar x_1^0&-&\bar x_5^0&\bar x_3^0&x_4^0&x_1^0
	\end{array}
	$$

As the fourth iteration commences, the only unsatisfied clause left is $\PHI_3$, whence $i_4=3$.
Moreover, assume that $j_4=1$.
As we have revealed all occurrences of $x_1,\ldots,x_6$, at this point we know that
$|\PHI_{31}|$ is uniformly distributed over $\cbc{x_7,x_8,x_9,x_{10}}$.
Suppose that indeed $|\PHI_{31}|=x_7$.
Thus, {\bf PI1} sets $\sigma_4(x_2)=\sigma_4(x_7)=0$ and $\sigma_4(x)=1$ for all $x\neq x_2,x_7$.
Suppose that revealing all occurrences of $x_7$ yields
	$$
	\textcolor{white}{\pi_3=}
	\begin{array}{cccccc}
	\bar x_2^0&-&{\bf \bar x_7}&x_1^0&+&{\bf x_7}\\
	\bar x_3^0&x_1^0&-&x_5^0&{\bf \bar x_7}&+\\
	\bar x_4^0&{\bf \bar x_7}&-&\bar x_6^0&\bar x_1^0&x_6^0\\
	\bar x_5^0&-&-&\bar x_4^0&x_3^0&x_2^0\\	
	\bar x_1^0&-&\bar x_5^0&\bar x_3^0&x_4^0&x_1^0
	\end{array}
	$$
Then there are no $\cA_3\cup\cN_3\cup\cbc{x_7}$-negative clauses $\PHI_i$ with $i\not\in\cZ_3$ that have at least two occurrences
of a variable from $\cA_3\cup\cbc{x_7}$.
Therefore, {\bf PI2} is void, and  at the end of the fourth iteration we have
	$$
	\textcolor{black}{\pi_4=}
	\begin{array}{cccccc}
	\bar x_2^0&-&\bar x_7^*&x_1^0&+&x_7^*\\
	\bar x_3^0&x_1^0&-&x_5^0&\bar x_7^*&+\\
	\bar x_4^0&\bar x_7^*&-&\bar x_6^0&\bar x_1^0&x_6^0\\
	\bar x_5^0&-&-&\bar x_4^0&x_3^0&x_2^0\\	
	\bar x_1^0&-&\bar x_5^0&\bar x_3^0&x_4^0&x_1^0
	\end{array},
	$$
$\cA_4=\cbc{x_7}$, $\cN_4=\cbc{x_1,\ldots,x_6}$, and $\cZ_4=\cbc{1,4}$.
As $\sigma_4$ is satisfying the process stops and $T=4$.
\qed
\end{example}

To 
trace the process {\bf PI0}--{\bf PI3} over time we define a filtration $(\cF_t)_{t\geq0}$ by letting $\cF_t$ be the $\sigma$-algebra
generated by the random variables $i_s,j_s$ and $\pi_s(i,j)$ with $s\leq t$ and $(i,j)\in\brk m\times\brk k$.
Then intuitively, a random variable $X$ is $\cF_t$-measurable if its value is determined by the first $t$ steps of the process {\bf PI0--PI3}.
In particular, we have the following.

\begin{fact}\label{Fact_measurable}
For any $t\geq1$, any $x\in V$, and any $i\in\brk m$ the events $\cbc{\sigma_t(x)=1}$,
	$\cbc{\mbox{$\PHI_i$ is satisfied under $\sigma_t$}}$,
		$\cbc{x\in\cA_t}$, $\cbc{i\in\cZ_t}$, $\cbc{x\in\cN_t}$, and $\cbc{T=t}$ are
$\cF_t$-measurable.
\end{fact}
\begin{proof}
The construction in steps {\bf PI2} and {\bf PI3} ensures that for any $t\geq1$ we have $\PHI_{i_tj_t}\in\cA_t\cup\cN_t$
	and thus $\pi_t(i_t,j_t)=\PHI_{i_tj_t}$
This implies that for any variable $x\in V$ the event $\cbc{\sigma_t(x)=1}$ is $\cF_t$-measurable.
In fact, we have $\sigma_t(x)=1$ iff the number $\abs{\cbc{1\leq s\leq t:|\pi_t\bc{i_s,j_s}|=x}}$ of times
	$x$ has been flipped is even 	(because $\sigma_0$ is the all-true assignment).

This implies that for any $i\in\brk m$ the event $\cbc{\mbox{$\PHI_i$ is satisfied under $\sigma_t$}}$ is $\cF_t$-measurable.
In fact, if there is an index $j\in\brk k$ such that $\pi_t(i,j)=1$, then $\PHI_{ij}$ is a positive literal whose underlying
variable has not been flipped before, whence $\sigma_t$ satisfies $\PHI_i$.
Moreover, if there is an index $j\in\brk k$ such that $\PHI_{ij}\neq\pm1$, then by the previous paragraph
the event that the literal $\PHI_{ij}=\pi_t(i,j)$ is true under $\sigma_t$ is $\cF_t$-measurable.
If there is such a satisfied literal $\PHI_{ij}$, then $\PHI_i$ is satisfied.
Conversely, if there is no $j\in\brk k$ such that either $\pi_t(i,j)=1$ or $\pi_t(i,j)$ is a literal that is satisfied under $\sigma_t$, then
clause $\PHI_i$ is unsatisfied.
Hence, the event $\cbc{\sigma_t\mbox{ is satisfying}}$ is $\cF_t$-measurable as well, and therefore
so is the event $\cbc{T=t}$.

Furthermore, observe that $i\in\cZ_t$ iff for all $j\in\brk k$ we have $\pi_t(i,j)\not\in\cbc{-1,1}$.
For if $i\in\cZ_t$, then for all $j\in\brk k$ we have $|\PHI_{ij}|\in\cN_t$ and thus $\pi_t(i,j)=\PHI_{ij}\neq\pm1$ due to {\bf PI3}.
Conversely, if $k\geq k_0$ is large enough, any $i\in\brk k$ such that $\pi_t(i,j)\not\in\cbc{-1,1}$
for all $j\in\brk k$ must satisfy one of the two conditions that lead {\bf PI2} to add $i$ to $\cZ_t$.
Hence, for any $i\in\brk m$ the event $\cbc{i\in\cZ_t}$ is $\cF_t$-measurable.
As by construction $\cN_t=\cbc{\pi_t(i,j):i\in\cZ_t,j\in\brk k}$, we conclude that for any variable $x\in V$
the event $\cbc{x\in\cN_t}$ is $\cF_t$-measurable.

Finally, the construction in {\bf PI3} ensures that $\cA_t=\cbc{|\pi_t(i_s,j_s)|:1\leq s\leq t}\setminus\cN_t$.
As for any $x$ the events $\cbc{x\in\cbc{|\pi_t(i_s,j_s)|:1\leq s\leq t}}$ and $\cbc{x\in\cN_t}$ are $\cF_t$-measurable,
so is the event $\cbc{x\in\cA_t}$.
\qed\end{proof}

If $\pi_t(i,j)=\pm1$, then up to time $t$ the process {\bf PI0}--{\bf PI3} has only
taken the sign of the literal $\PHI_{ij}$ into account, but has been oblivious to the underlying variable.
The only conditioning is that $|\PHI_{ij}|\not\in \cA_t\cup\cN_t$ (because otherwise {\bf PI3} would have replaced the $\pm1$ by the actual literal).
Since the input formula $\PHI$ is random, this implies that $|\PHI_{ij}|$ is uniformly distributed over $V\setminus(\cA_t\cup\cN_t)$.
In fact, for all $(i,j)$ such that $\pi_t(i,j)=\pm1$ the underlying variables are independently uniformly distributed over $V\setminus(\cA_t\cup\cN_t)$.
Formally, we can state this key observation as follows.

\begin{fact}\label{Fact_iid}
Let $t\geq0$.
Let $\cE_t$ be the set of all pairs $(i,j)$ such that $\pi_t(i,j)\in\{-1,1\}$.
The conditional joint distribution of the variables $(|\PHI_{ij}|)_{(i,j)\in\cE_t}$ given $\cF_t$ is uniform over $(V\setminus(\cA_t\cup\cN_t))^{\cE_t}$.
That is, 
for any map $f:\cE_t\rightarrow V\setminus (\cA_t\cup\cN_t)$ we have
	$$\pr\brk{\forall (i,j)\in\cE_t:|\PHI_{ij}|=f(i,j)|\cF_t}=|V\setminus(\cA_t\cup\cN_t)|^{-|\cE_t|}.$$
\end{fact}

Let
	$$T^*=\theta n\qquad\mbox{with }\theta=0.38/k.$$
Our overall goal is to prove that the stopping time of the process {\bf PI0--PI3} satisfies
	$T\leq T^*$ \whp\
(The number $\theta$ is chosen somewhat arbitrarily; for the analysis to work it seems to be essential that $\theta=c/k$ for some $c>0$ that is neither ``too small''
		nor ``too large''.
	The concrete constant above happens to work.)
To prove this, we will define non-negative random variables $S_t,H_t$
such that $S_t+H_t=0$ implies that $\sigma_t$ is a satisfying assignment.
We will then trace $S_t,H_t$ for $1\leq t\leq T^*$.

For any $t\geq1$ let
	$$\cD_t=\cbc{i\in\brk m:\mbox{$\PHI_i$ is $\cA_t\cup\cN_t$-negative}}.$$
As {\bf PI3} ensures that $\PHI_i$ is $\cA_t\cup\cN_t$-negative iff $\pi_t(i,j)\neq1$ for all $j\in\brk k$,
the event $\cbc{i\in\cD_t}$ is $\cF_t$-measurable for any $i\in\brk m$.
We define
	\begin{eqnarray}\label{eqSt}
	S_0=|\cD_0|&\mbox{and}&S_t=|\cD_t|-\abs{\cA_t}\mbox{ for }t\geq1.
	\end{eqnarray}
Any clause $\PHI_i$ with $i\not\in\cD_t$ is satisfied under $\sigma_t$.
	For if $j\in\brk k$ is such that $\pi_t(i,j)=1$, then $\PHI_{ij}$ is a positive literal and $\sigma_t(\PHI_{ij})=1$,
		because \Walksat\ starts with the all-true assignment $\sigma_0$ and the variable $\PHI_{ij}$ has not been flipped up to time $t$.
Clearly, in order to study the random variable $S_t$ it is crucial to estimate $\abs{\cD_t}$.
This is the purpose of the following proposition, whose proof we defer to \Sec~\ref{Sec_D}.

\begin{proposition}\label{Prop_D}
\Whp\ we have $|\cD_t|\leq2^{2-k}m$ for all $t\leq T^*$.
\end{proposition}

To define the random variables $H_t$, let us
call an assignment $\tau:\cN_t\ra\cbc{0,1}$ \emph{rich} for $\cZ_t$
if in each clause $\PHI_i$ with $i\in\cZ_t$ at least $0.8k$ literals $\PHI_{ij}$ are satisfied under $\tau$.

\begin{proposition}\label{Prop_tau}
\Whp\ there is a sequence $(\tau_t)_{1\leq t\leq T^*}$ with the following properties.
\begin{enumerate}
\item For any $1\leq t\leq T^*$, $\tau_t$ is a rich assignment for $\cZ_t$.
\item For any $1<t\leq T^*$ and any $x\in\cN_{t-1}$ we have $\tau_t(x)=\tau_{t-1}(x)$.
\end{enumerate}
Moreover, $\tau_t$ is $\cF_t$-measurable for all $t$.
\end{proposition}
Assuming that there is a sequence $(\tau_t)_{1\leq t\leq T^*}$ as in \Prop~\ref{Prop_tau}, we define $H_0=0$ and
	$$H_t=\abs{\cbc{x\in\cN_t:\sigma_t(x)\neq\tau_t(x)}}\mbox{ for }1\leq t\leq T^*,$$
and $H_t=\abs{\cN_t}$ for $t>T^*$.
For the sake of completeness, we also let $H_t=|\cN_t|$ if there is no such sequence $(\tau_t)_{1\leq t\leq T^*}$.
The proof of \Prop~\ref{Prop_tau} hinges upon the following fact.

\begin{proposition}\label{Prop_Z}
\Whp\ we have $\abs{\cZ_t}\leq \eps n$ for all $t\leq T^*$.
\end{proposition}
We defer the proof of \Prop~\ref{Prop_Z} to \Sec~\ref{Sec_Z}.
Assuming \Prop~\ref{Prop_Z}, we can derive \Prop~\ref{Prop_tau} rather easily.

\medskip\noindent\emph{Proof of \Prop~\ref{Prop_tau} (assuming \Prop~\ref{Prop_Z}).}
By \Lem~\ref{Lemma_Hall}, we may assume that $\PHI$ has the expansion property~(\ref{eqHall}).
Furthermore, by \Prop~\ref{Prop_Z} we may assume that $\abs{\cZ_t}\leq\eps n$ for all $t\leq T^*$.
Under these assumptions we will construct the sequence $(\tau_t)_{1\leq t\leq T^*}$ by induction on $t\geq1$.
Thus, suppose that $1\leq t\leq T^*$ and that we have already got assignments $\tau_s$ with $1\leq s<t$ that satisfy 1.--2.

The set $Z=\cZ_t\setminus\cZ_{t-1}$ 
	of indices that $\cZ_t$ gained at time $t$
		has size $|Z|\leq\abs{\cZ_t}\leq\eps n$.
Therefore, (\ref{eqHall}) ensures that there is a $0.9k$-fold matching $M$ from $Z$ to the set
	$$N=N(\PHI_Z)=\cbc{|\PHI_{ij}|:(i,j)\in Z\times\brk k}\subset\cN_t$$
of variables that occur in the clauses $\PHI_i$ with $i\in Z$.
The construction in {\bf PI2} ensures that none of these clauses $\PHI_i$ has more than $\lambda$ occurrences of a variable from $\cN_{t-1}$
		(as otherwise $i\in\cZ_{t-1}$).
Therefore, in the matching $M'$ obtained from $M$ by omitting all edges $e=\cbc{i,x}$ with $i\in Z$ and $x\in\cN_{t-1}$
each clause $\PHI_i$ with $i\in Z$ is incident with at least $0.9k-\lambda\geq0.8k$ edges. 
Now, for each edge $e=\cbc{i,x}\in M'$ let $\tau_t(x)$ be the truth value that
makes the corresponding literal in $\PHI_i$ evaluate to true.
Furthermore, for all $y\in\cN_{t-1}$ let $\tau_t(y)=\tau_{t-1}(y)$,
and for all other variables $x'\in\cN_t$ let $\tau_t(x')=1$.
This ensures that $\tau_t$ satisfies the conditions in \Prop~\ref{Prop_tau}.
\qed

Having defined the random variables $S_t,H_t$, we are now going to verify
that they suit their intended purpose, i.e., that $S_t+H_t=0$ implies that $\sigma_t$ is satisfying.

\begin{proposition}\label{Prop_U}
Let $1\leq t\leq T^*$.
If $S_t+H_t=0$, then $\sigma_t$ is a satisfying assignment.
\end{proposition}
\begin{proof}
Let $U_t$ be the number of clause indices $i\in\brk m\setminus\cZ_t$ such that $\PHI_i$ is unsatisfied under $\sigma_t$.
We claim that
	\begin{eqnarray}\label{eqU1}
	U_t&\leq&S_t=|\cD_t|-|\cA_t|.
	\end{eqnarray}
To see this, recall that any index $i\in\brk m$ such that $\PHI_i$ is unsatisfied under $\sigma_t$ belongs to $\cD_t$.
Therefore, to prove~(\ref{eqU1}) it suffices to construct injective maps $s_t:\cA_t\ra\cD_t$ such that for any $x\in\cA_t$ the clause
$\PHI_{s_t(x)}$ is satisfied under $\sigma_t$.
In fact, the map $s_t$ will have the property that for each $x\in\cA_t$ there is an index $j\in\brk k$ such that $x=|\PHI_{s_t(x)j}|$ and such
that the literal $\PHI_{s_t(x)j}$ is true under $\sigma_t$.

The construction of the maps $s_t$ is inductive.
For $t=0$ we have $\cA_0=\emptyset$ and thus there is nothing to do.
Thus, suppose that $1\leq t\leq T$ and that we have defined $s_{t-1}$ already.
Let $y=|\PHI_{i_tj_t}|$ be the variable flipped at time $t$. 
If $i_t\not\in\cZ_t$, then $y\in\cA_t$ and we define $s_t(y)=i_t$.
Moreover, we let $s_t(x)=s_{t-1}(x)$ for all $x\in\cA_t\setminus\cbc y\subset\cA_{t-1}$.
(Note that it is possible that $y\in\cA_{t-1}$ as $y$ may have been flipped before.)
For $t>T$ we set $s_t=s_{t-1}$.

To verify that $s_t$ has the desired properties, assume that $T\geq t$ and observe that
	 {\bf PI1} ensures that $\PHI_{i_t}$ was unsatisfied under $\sigma_{t-1}$.
Thus, $i_t\in\cD_{t-1}\subset\cD_t$.
But as {\bf PI1} sets $\sigma_t(y)=1-\sigma_{t-1}(y)$, $\PHI_{i_t}$ is satisfied under $\sigma_t$. 
Furthermore, for all $x\in\cA_{t}\setminus\cbc y$ we have $\sigma_t(x)=\sigma_{t-1}(x)$, and thus
	each of these variables contributes a true literal to its clause $\PHI_{s_t(x)}=\PHI_{s_{t-1}(x)}$ by induction.
Since $s_{t-1}$ is injective but $\PHI_{i_t}$ was unsatisfied under $\sigma_{t-1}$, we have
	$i_t\not\in\mathrm{Im}(s_{t-1})$, whence $s_t$ is injective.
This establishes~(\ref{eqU1}).

As~(\ref{eqU1}) shows, $S_t=0$ implies $U_t=0$, i.e., $\sigma_t$ satisfies all clauses $\PHI_i$ with $i\not\in\cZ_t$.
To complete the proof, we need to show that if $H_t=0$, then $\sigma_t$ also satisfies all clauses $\PHI_i$ with $i\in\cZ_t$.
But if $H_t=0$, then $\sigma_t(x)=\tau_t(x)$ for all $x\in\cN_t$, and $\tau_t$ is a satisfying assignment of $\PHI_{\cZ_t}$.
\qed\end{proof}

Finally, we have all the pieces in place to prove \Thm~\ref{Thm_pos}.

\medskip
\noindent\emph{Proof of \Thm~\ref{Thm_pos} (assuming \Prop s~\ref{Prop_D} and~\ref{Prop_Z}).}
\Prop~\ref{Prop_U} shows that
	$$\pr\brk{T\geq T^*}=\pr\brk{T\geq T^*\wedge\forall 1\leq t\leq T^*:S_t+H_t>0}.$$
We are going to bound the probability on the r.h.s.
To this end, we work with two random variables $S_t',H_t'$ that are easier to analyze than the original $S_t,H_t$.
Namely, we let $S_0'=H_0'=0$, and 
	$$S_t'=S_{t-1}'-\left\{\begin{array}{cl}
			1&\mbox{ if }\pi_{t-1}(i_t,j_t)=-1,\\
			0&\mbox{ otherwise}
			\end{array}\right.
			\qquad\qquad\qquad\qquad\qquad(t\geq1).$$
In other words, we let $S_t'=S_{t-1}'-1$ if the variable flipped at time $t$
	had not been flipped before and does not occur in any of the `exceptional' clauses $\PHI_{\cZ_{t-1}}$.
Otherwise, $S_t'=S_{t-1}'$.

We claim that
	\begin{eqnarray}\label{eqpos1}
	S_t&\leq&\abs{\cD_t}+k\abs{\cZ_t}+S_t'\qquad\mbox{for any }t\geq0.
	\end{eqnarray}
To see this, recall from~(\ref{eqSt}) that $S_t=\abs{\cD_t}-\abs{\cA_t}$.
By {\bf PI3}, the set $\cA_t$ contains all variables $|\PHI_{i_sj_s}|$ such that $\pi_{s-1}(i_s,j_s)=-1$ with $s\leq t$, except
the ones that belong to $\cN_t$.
Since $|\cN_t|\leq k\abs{\cZ_t}$, we obtain~(\ref{eqpos1}).

Furthermore, we let $H_0'=0$ and 
	$$H_t'=H_{t-1}'+
			\left\{\begin{array}{cl}
			-1&\mbox{ if $|\PHI_{i_tj_t}|\in\cN_{t-1}$ and $\sigma_{t}(|\PHI_{i_tj_t}|)=\tau_{t}(|\PHI_{i_tj_t}|)$,}\\
			1&\mbox{ if $|\PHI_{i_tj_t}|\in \cN_{t-1}$ and $\sigma_{t}(|\PHI_{i_tj_t}|)\neq\tau_{t}(|\PHI_{i_tj_t}|)$,}\\
			0&\mbox{ otherwise}
			\end{array}\right.
				\qquad\qquad\qquad(t\geq1).$$
Thus, starting at $0$, we decrease the value of $H_t'$ by one if the variable flipped at time $t$ lies in $\cN_{t-1}$
and its new value coincides with the `ideal' assignment $\tau_t$, while we increase by one if these values differ.

We claim that 
	\begin{eqnarray}\label{eqpos2}
	H_t&\leq&k\abs{\cZ_t}+H_t'\qquad\mbox{for any }t\geq0.
	\end{eqnarray}
For $H_0=H_0'$ and
	\begin{eqnarray*}
	H_t-H_{t-1}&=&\abs{\cbc{x\in\cN_t:\sigma_t(x)\neq\tau_t(x)}}-
			\abs{\cbc{x\in\cN_{t-1}:\sigma_{t-1}(x)\neq\tau_{t-1}(x)}}\\
		&\leq&\abs{\cN_t\setminus\cN_{t-1}}+H_t'-H_{t-1}'
			\leq k\abs{\cZ_t\setminus\cZ_{t-1}}+H_t'-H_{t-1}'\qquad\qquad\mbox{ for any }t\geq1.
	\end{eqnarray*}

Combining~(\ref{eqpos1}) and~(\ref{eqpos2}) with \Prop s~\ref{Prop_D} and~\ref{Prop_Z}, we see that \whp
	\begin{eqnarray}\nonumber
	S_t+H_t&\leq&\abs{\cD_t}+2k\abs{\cZ_t}+S_t'+H_t'\\
			&\leq&2^{2-k}m+2k\abs{\cZ_t}+S_t'+H_t'\leq\frac{4\rho n}{k}+2k\eps n+S_t'+H_t'
				\qquad\mbox{for any }t\leq T^*.
				\label{eqpos3}
	\end{eqnarray}
Hence, we are left to analyze $S_t'+H_t'$.

The sequence $(S_t'+H_t')_t$ is a super-martingale.
More precisely, we claim that with $\gamma=0.429$ we have
	\begin{eqnarray}\label{eqpos4}
	\Erw\brk{S_t'+H_t'|\cF_{t-1}}<S_{t-1}'+H_{t-1}'-\gamma\qquad\mbox{for all }t\leq\min\cbc{T,T^*}.
	\end{eqnarray}
There are two cases to consider.
\begin{description}
\item[Case 1: $i_t\not\in\cZ_{t-1}$.]
	The construction in step {\bf PI2} ensures that there are fewer than $\lambda$ indices $j$
	such that $|\PHI_{i_tj}|\in\cN_{t-1}$.
	Furthermore, {\bf PI2} ensures that there are less than $k_1$ indices $j$ such that
		$|\PHI_{i_tj}|\in \cA_{t-1}$.
	Moreover, there is no index $j$ such that $\pi_{t-1}(i_t,j)=1$, because
	otherwise clause $\PHI_{i_t}$ would have been satisfied under $\sigma_{t-1}$.
	This means that for at least $k-k_1-\lambda$ indices $j\in\brk k$ we have
		$\pi_{t-1}(i_t,j)=-1$.
	Therefore, as $j_t\in\brk k$ is chosen uniformly at random,
			with probability at least $1-(k_1+\lambda)/k\geq0.43-\lambda/k$ we have $S_t'=S_{t-1}'-1$.
	In addition, as $\PHI_{i_t}$ contains at most $\lambda$ variables from $\cN_{t-1}$,
	the probability that $H_t'=H_{t-1}'+1$ is bounded from above by $\lambda/k<0.00001$.
	Thus, (\ref{eqpos4}) holds. 
\item[Case 2: $i_t\in\cZ_{t-1}$.]
	As the assignment $\tau_{t-1}$ is rich, there are at least $0.8k$ indices $j$ such that $\tau_t(\PHI_{i_tj})=\tau_{t-1}(\PHI_{i_tj})=1$.
	However, for all of these indices $j$ we have $\sigma_{t-1}(\PHI_{i_tj})=0$, because $\PHI_{i_t}$ is unsatisfied under $\sigma_{t-1}$.
	Hence, the probability that $\tau_t(\PHI_{i_tj_t})=1$ and $\sigma_{t-1}(\PHI_{i_tj_t})=0$ is at least $0.8$,
		and if this event indeed occurs then $\sigma_t(\PHI_{i_tj_t})=\tau_t(\PHI_{i_tj_t})=1$.
	Therefore,
		$H_t'-H_{t-1}'$ has expectation $\leq-0.8+0.2\leq-0.6$.
	Moreover, $S_t'\leq S_{t-1}'$ with certainty.
	This implies~(\ref{eqpos4}). 
\end{description}

To complete the proof, we are going to apply Azuma's inequality (\Lem~\ref{Azuma} in \Sec~\ref{Sec_pre})
to the random variable $S_{T^*}'+H_{T^*}'$.
The inequality applies because~(\ref{eqpos4}) shows that $(S_t'+H_t')_{t\geq0}$ is a super-martingale.
However, there is a minor technical intricacy: to use the inequality,
we need an upper bound on the \emph{expectation} $\Erw\brk{S_{T^*}'+H_{T^*}'}$.
But as~(\ref{eqpos4}) only holds for $t\leq\min\cbc{T,T^*}$,
this would require knowledge of the probability that $T\geq T^*$,
the very quantity that we want to estimate.

To circumvent this problem, we define further random variables $R_t$ by letting $R_t=S_t'+H_t'$ for $t\leq\min\cbc{T^*,T}$
and $R_t=R_{t-1}-\gamma$ for $t>\min\cbc{T^*,T}$.
Then $R_0=0$ and $\Erw\brk{R_t|\cF_{t-1}}\leq R_{t-1}-\gamma$ for all $t\geq0$.
Thus, 
	$\Erw\brk{R_{T^*}}\leq-\gamma\,T^*$.
Recalling the definition~(\ref{eqlambdaeps}) of $\eps$, we obtain
for $k\geq k_0$ sufficiently large and $\rho\leq\rho_0=1/25$ 
the bound 
	\begin{equation}\label{eqRTexp}
	\Erw\brk{R_{T^*}}\leq-\gamma\cdot T^*\leq-4\rho n/k-10k\eps n.
	\end{equation}
Furthermore, $|R_t-R_{t-1}|\leq2$ for all $t\geq0$ by the definitions of $S_t',H_t'$.
Therefore, Azuma's inequality and~(\ref{eqRTexp}) yield
	\begin{eqnarray}\label{eqRT}
	\pr\brk{R_{T^*}>-4\rho n/k-2k\eps n}&\leq&\pr\brk{R_{T^*}>\Erw\brk{R_{T^*}}+n^{2/3}}
		\leq\exp\brk{-\frac{n^{4/3}}{8T^*}}=o(1).
	\end{eqnarray}
Finally, we obtain from~(\ref{eqpos1}), (\ref{eqpos2}), and \Prop~\ref{Prop_U}
	\begin{eqnarray*}
	\pr\brk{T>T^*}&\leq&	\pr\brk{\forall t\leq T^*:\abs{\cD_t}+2k\abs{\cZ_t}+R_t>0}
		\leq\pr\brk{\abs{\cD_{T^*}}+2k\abs{\cZ_{T^*}}+R_{T^*}>0}\\
		&\leq&\pr\brk{\abs{\cD_{T^*}}+2k\abs{\cZ_{T^*}}>4\rho n/k+2k\eps n}+\pr\brk{R_{T^*}>-4\rho n/k-2k\eps n}
		\;\stacksign{(\ref{eqpos3}),\,(\ref{eqRT})}{=}\;o(1),
	\end{eqnarray*}
thereby completing the proof.
\qed

Our remaining task is to establish \Prop s~\ref{Prop_D} and~\ref{Prop_Z}.
From a formal point of view, we should start with \Prop~\ref{Prop_Z} because the proof of \Prop~\ref{Prop_D} depends on it.
However, the argument that is used in the proof of \Prop~\ref{Prop_D} is conceptually similar to but technically far simpler
than the one that we use to prove \Prop~\ref{Prop_Z}.
Hence, for didactical reasons we will start with the proof of \Prop~\ref{Prop_D} in \Sec~\ref{Sec_D}
and postpone the proof of \Prop~\ref{Prop_Z} to \Sec~\ref{Sec_Z}.

\section{Proof of \Prop~\ref{Prop_D}}\label{Sec_D}

\emph{In this section we keep the notation and the assumptions from \Prop~\ref{Prop_D}.}

Our goal is to bound the number $|\cD_{T^*}|$ of $\cA_{T^*}\cup\cN_{T^*}$-negative clauses $\PHI_i$,
i.e., clauses whose positive literals all belong to $\cA_{T^*}\cup\cN_{T^*}$.
Thus, we need to study how the process {\bf PI0--PI3} `hits' the positions $(i,j)\in\brk m\times\brk k$
that represent positive literals by adding their underlying variable
to $\cA_{T^*}\cup\cN_{T^*}$.
To this end, we consider the two random variables
	\begin{eqnarray}\label{eqKt*}
	K_t^*(i,j)&=&\left\{\begin{array}{cl}
			1&\mbox{ if }
					\pi_{t-1}(i,j)=1
						\mbox{ and }\PHI_{ij}\in\cA_t,\\
			0&\mbox{ otherwise,}
			\end{array}\right.\\
	K_t^0(i,j)&=&\left\{\begin{array}{cl}
			1&\mbox{ if }
					\pi_{t-1}(i,j)=1\mbox{ and }\PHI_{ij}\in\cN_t,\\
			0&\mbox{ otherwise,}
			\end{array}\right.
			\label{eqKt0}
	\end{eqnarray}
for any $(i,j)\in\brk m\times\brk k$ and $t\geq1$.
Recall that $\pi_{t-1}(i,j)=\sign(\PHI_{ij})$ iff $\PHI_{ij}$ is a literal such that $|\PHI_{ij}|\not\in\cA_{t-1}\cup\cN_{t-1}$
	(cf.\ {\bf PI3}).
To simplify the notation, we define for a set $\cI\subset\brk m\times\brk k$
	$$K_t^*(\cI)=\prod_{(i,j)\in\cI}K_t^*(i,j),\qquad K_t^0(\cI)=\prod_{(i,j)\in\cI}K_t^0(i,j).$$
If $\cI^*,\cI^0\subset\brk m\times\brk k$ are both non-empty, then
	\begin{equation}\label{eqKt*Kt0}
	K_t^*(\cI^*)\cdot K_t^0(\cI^0)=0.
	\end{equation}
Indeed, suppose that $K_t^0(\cI^0)\neq0$.
Then {\bf PI2} must have added at least one clause to $\cZ_t$.
But the construction in {\bf PI2} ensures that the first clause that gets added to $\cZ_t$ 
contains the variable $|\PHI_{i_tj_t}|$ flipped at time $t$.
Thus, $\cA_t\subset\cA_{t-1}$ by {\bf PI3}, and thus there cannot be a pair $(i,j)$ with $K_t^*(i,j)=1$.
In effect, $K_t^*(\cI^*)=0$.

\begin{lemma}\label{Lemma_Kt*}
Let $t\geq1$ and  $\emptyset\neq\cI^*\subset\brk m\times\brk k$.
Let $\cE_t^*(\cI^*)$ be the event that $|\PHI_{ij}|=|\PHI_{i_tj_t}|\not\in\cA_{t-1}\cup\cN_{t-1}$ for all $(i,j)\in\cI^*$, and that $(i_t,j_t)\not\in\cI^*$.
Then
	\begin{eqnarray}\label{eqKt8}
	\pr\brk{\cE_t^*(\cI^*)|\cF_{t-1}}&\leq&\max\cbc{1,|V\setminus(\cA_{t-1}\cup\cN_{t-1})|}^{-\abs{\cI^*}}.\ 
	\end{eqnarray}
\end{lemma}
\begin{proof}
Since clause $\PHI_{i_t}$ is unsatisfied under $\sigma_{t-1}$, $\PHI_{i_t}$ is $\cA_{t-1}\cup\cN_{t-1}$-negative
	and thus $\pi_{t-1}(i_t,j_t)\neq1$.
Hence, {\bf PI3} ensures that either $|\PHI_{i_tj_t}|\in\cA_{t-1}\cup\cN_{t-1}$ or $\pi_{t-1}(i_t,j_t)=-1$.
If $\cE_t^*(\cI^*)$ occurs, then $|\PHI_{i_tj_t}|\not\in\cA_{t-1}\cup\cN_{t-1}$ and thus $\pi_{t-1}(i_t,j_t)=-1$.
Furthermore, if $\cI^*$ occurs, then $|\PHI_{ij}|\not\in\cA_{t-1}\cup\cN_{t-1}$ for all $(i,j)\in\cI^*$,
and thus $\pi_{t-1}(i,j)\in\cbc{-1,1}$ by {\bf PI3}.
Thus, by Fact~\ref{Fact_iid} $|\PHI_{i_tj_t}|$ and $|\PHI_{ij}|$ with $(i,j)\in\cI^*$ are independently uniformly distributed over
$V\setminus(\cA_{t-1}\cup\cN_{t-1})$. 
Therefore,
	\begin{eqnarray*}
	\pr\brk{\cE_t^*(\cI^*)|\cF_{t-1}}&\leq&\max\cbc{1,|V\setminus(\cA_{t-1}\cup\cN_{t-1})|}^{-|\cI^*|},
	\end{eqnarray*}
as claimed.
\qed\end{proof}

\begin{corollary}\label{Cor_Kt*}
For any $t\geq1$,  $\cI^*\subset\brk m\times\brk k$ we have
	\begin{eqnarray*}
	\Erw\brk{K_t^*(\cI^*)|\cF_{t-1}}&\leq&\max\cbc{1,|V\setminus(\cA_{t-1}\cup\cN_{t-1})|}^{-\abs{\cI^*}}.\ 
	\end{eqnarray*}
\end{corollary}
\begin{proof}
If $\prod_{(i,j)\in\cI^*}K_t^*(i,j)=1$, then the event $\cE_t^*(\cI^*)$ occurs.
Hence, \Lem~\ref{Lemma_Kt*} implies that
	\begin{eqnarray}
	\Erw\brk{\prod_{(i,j)\in\cI^*}K_t^*(i,j)|\cF_{t-1}}&\leq&\pr\brk{\cE_t^*(\cI^*)|\cF_{t-1}}\leq
		\max\cbc{1,|V\setminus(\cA_{t-1}\cup\cN_{t-1})|}^{-|\cI^*|},
	\end{eqnarray}
as claimed.
\qed\end{proof}

\begin{lemma}\label{Lemma_Kt0}
For any $t\geq1$, $\delta_t\geq0$ and $\cI^0\subset\brk m\times\brk k$ we have
	\begin{eqnarray*}
	\Erw\brk{
		K_t^0(\cI^0)\cdot
				\vecone\cbc{|\cZ_t\setminus\cZ_{t-1}|\leq\delta_t}|\cF_{t-1}}
			&\leq&
					\bcfr{k\delta_t}{\max\cbc{1,|V\setminus(\cA_{t-1}\cup\cN_{t-1})|-k\delta_t}}^{\abs{\cI^0}}.
	\end{eqnarray*}
\end{lemma}
\begin{proof}
We may  assume that $\cI^0\neq\emptyset$.
We may also assume that
	$\pi_{t-1}(i,j)=1$ for all $(i,j)\in\cI^0$ as otherwise 
	$K_t^0(\cI^0)=0$.
We are going to work with the conditional distribution
	$$p\brk{\cdot}=\pr\brk{\cdot|\cF_{t-1}}.$$
Let $\cE^0$ be the event that $K_t^0(\cI^0)=1$ and $|\cZ_t\setminus\cZ_{t-1}|\leq\delta_t$.
Then our goal is to estimate 
	$p\brk{\cE^0}$.

If the event $\cE^0$ occurs, then $\pi_{t-1}(i_t,j_t)=-1$ and $|\PHI_{i_tj_t}|\in\cN_t$.
Indeed, being unsatisfied under the assignment $\sigma_{t-1}$, clause $\PHI_{i_t}$ is $\cA_{t-1}\cup\cN_{t-1}$-negative,
and thus $\pi_{t-1}(i_t,j_t)\neq1$.
Furthermore, if $\pi_{t-1}(i_t,j_t)=\PHI_{i_tj_t}$, then $|\PHI_{i_tj_t}|\in\cA_{t-1}\cup\cN_{t-1}$ by {\bf PI3},
and thus $\cZ_t=\cZ_{t-1}$ and $\cN_t=\cN_{t-1}$ by the construction in step {\bf PI2}.
But if $\cN_t=\cN_{t-1}$, then $K_t^0(\cI^0)=0$ by definition.

Thus, assume that $\pi_{t-1}(i_t,j_t)=-1$ and $|\PHI_{i_tj_t}|\in\cN_t$.
We need to trace the process described in {\bf PI2} that enhances the sets $\cN_t$ and $\cZ_t$.
This process may add a sequence of clause indices to the set $\cZ_t$ and the variables that
occur in these clauses to $\cN_t$.
As these variables get added to the set $\cN_t$ one by one, we will study the probability
that they occur in one of the positions $(i,j)\in\cI^0$.
The first clause that {\bf PI2} adds to $\cZ_t$  necessarily contains the newly flipped variable $|\PHI_{i_tj_t}|$, and thus
we may assume that this is the first variable that gets added to $\cN_t$.
In addition, if $\abs{\cZ_t\setminus\cZ_{t-1}}\leq\delta_t$, {\bf PI2} may add up to $k\delta_t-1$ further variables  to $\cN_t$.
To track this process, we need a bit of notation.

Let $s_1,\ldots,s_{y}$ 
be the clause indices 
that {\bf PI2} adds to $\cZ_t$, in the order in which they get added by the process.
Let $y^*=\min\cbc{y,\delta_t}$.
For each $1\leq i\leq y^*$ let $1\leq j_{i,1}<\cdots<j_{i,l_i}\leq k$ be the unique sequence of indices such that
	$\pi_{t-1}(s_i,j_{i,q})=-1$ and
		$$|\PHI_{s_i j_{i,q}}|\not\in\cbc{|\PHI_{i_tj_t}|}\cup\cN_{t-1}\cup\bigcup_{h=1}^{i-1} N(\PHI_{s_h})\cup\cbc{|\PHI_{s_i j_{i,u}}|:u<q}
			\mbox{ for all }q\leq l_i.$$
This means that $\cbc{|\PHI_{s_i j_{i,q}}|:1\leq q\leq l_i}$ are the new variables that $\PHI_{s_i}$ contributes to $\cN_t$
	and that did not belong to $\cA_{t-1}$ already.
Let $\xi_0=|\PHI_{i_tj_t}|$ and let $\xi_1,\ldots,\xi_L$ be the sequence of variables
	$|\PHI_{s_i j_{i,q}}|$ with $q=1,\ldots,l_i$ and $i=1,\ldots,y^*$.
Hence, $\xi_0,\ldots,\xi_L$ is the sequence of variables not in $\cA_{t-1}$ that {\bf PI2} adds to $\cN_t$,
	in the order in which the process adds these variables to $\cN_t$.
By our choice of $y^*$, the total number of these variables satisfies
	$$L+1\leq ky^*\leq k\delta_t.$$
Of course, $L$ and $\xi_0,\ldots,\xi_L$ are random variables.

If $\cE^0$ occurs, then 
each of the variables $\PHI_{ij}$ with $(i,j)\in\cI^0$ occurs in the sequence $\xi_0,\ldots,\xi_L$.
Hence, there exists a map $f:\cI^0\ra\cbc{0,1,\ldots,k\delta_t-1}$ such that $f(i,j)\leq L$ and
$\PHI_{ij}=\xi_{f(i,j)}$ for all $(i,j)\in\cI^0$.
For a given $f$ let $\cE^0(f)$ denote this event.
Then by the union bound,
	\begin{eqnarray}\label{eqKt9}
	p\brk{\cE^0}&\leq&\sum_{f:\cI^0\ra\cbc{0,1,\ldots,k\delta_t-1}}p\brk{\cE^0(f)}
			\leq(k\delta_t)^{|\cI^0|}\max_{f:\cI^0\ra\cbc{0,1,\ldots,k\delta_t-1}}p\brk{\cE^0(f)}.
	\end{eqnarray}

We claim that
	\begin{eqnarray}\label{eqKt10}
	p\brk{\cE^0(f)}&\leq&\max\cbc{1,|V\setminus(\cA_{t-1}\cup\cN_{t-1})|-k\delta_t}^{-|\cI^0|}
	\end{eqnarray}
for any $f$.
To prove~(\ref{eqKt10}), let
$\cI^0_l=f^{-1}(l)$ be the set of positions $(i,j)\in\cI^0$ where the variable $\xi_l$ occurs
 ($0\leq l\leq L$).
Moreover, let $\cE^0_l(f)$ be the event that 
\begin{enumerate}
\item[a.] $\PHI_{ij}=\xi_l$ for all $(i,j)\in\cI^0_l$, and
\item[b.] $\PHI_{ij}\neq\xi_l$ for all $(i,j)\in\cI^0\setminus\cI^0_l$. 
\end{enumerate}
As $\pi_{t-1}(i,j)=1$ for all $(i,j)\in\cI^0_l$, given $\cF_{t-1}$ the variables $\PHI_{ij}$ with $(i,j)\in\cI^0_l$
are independently uniformly distributed over $V\setminus(\cA_{t-1}\cup\cN_{t-1})$ by Fact~\ref{Fact_iid}.
Hence, given the event $\bigcap_{\nu<l}\cE^0_\nu(f)$,
the variables $|\PHI_{ij}|$ with $(i,j)\in\cI^0_l$ are uniformly distributed over the set
$V\setminus(\cA_{t-1}\cup\cN_{t-1}\cup\cbc{\xi_0,\ldots,\xi_{l-1}})$
(for if $\cE_\nu^0(f)$ occurs for some $\nu<l$, then $\PHI_{ij}\neq\xi_\nu$ for all $(i,j)\in\cI^0_l$).
Therefore, we obtain
	\begin{eqnarray*}
	p\brk{\cE^0_l(f)|\bigcap_{\nu<l}\cE^0_\nu(f)}&\leq&\max\cbc{1,|V\setminus(\cA_{t-1}\cup\cN_{t-1})|-l+1}^{-|\cI^0_l|}
		\quad\mbox{ for any }0\leq l\leq L.
	\end{eqnarray*}
Multiplying these conditional probabilities up for $0\leq l\leq L<k\delta_t$, we obtain~(\ref{eqKt10}).
Finally, combining~(\ref{eqKt8}), (\ref{eqKt9}), and~(\ref{eqKt10}) completes the proof.
\qed\end{proof}

\begin{corollary}\label{Cor_Kt}
For any $t\geq1$, $\delta_t\geq0$ and $\cI^*,\cI^0\subset\brk m\times\brk k$ we have
	\begin{eqnarray*}
	\Erw\brk{K_t^*(\cI^*)K_t^0(\cI^0)\vecone\cbc{|\cZ_t\setminus\cZ_{t-1}|\leq\delta_t}|\cF_{t-1}}\\
		&\hspace{-10cm}\leq&\,\,\,
		\hspace{-5cm}\max\cbc{1,|V\setminus(\cA_{t-1}\cup\cN_{t-1})|}^{-\abs{\cI^*}}
		\cdot\bcfr{k\delta_t}{\max\cbc{1,|V\setminus(\cA_{t-1}\cup\cN_{t-1})|-k\delta_t}}^{\abs{\cI^0}}.
	\end{eqnarray*}
\end{corollary}
\begin{proof}
This is immediate from
(\ref{eqKt*Kt0}) and \Cor~\ref{Cor_Kt*} and \Lem~\ref{Lemma_Kt0}.
\qed\end{proof}

Why does the bound provided by \Cor~\ref{Cor_Kt} ``make sense''?
First, observe that the only reason we need to take the max of the respective expression and one is because a priori it could happen that,
	e.g., $V\setminus(\cA_{t-1}\cup\cN_{t-1})=\emptyset$.
Apart from this issue, the first factor basically comes from the fact that for each pair $(i,j)$ with $\pi_{t-1}(i,j)=1$ the variable $\PHI_{ij}$ is
uniformly distributed over $V\setminus(\cA_{t-1}\cup\cN_{t-1})$.
Hence, it seems reasonable that the probability that one such $\PHI_{ij}$ equals the variable flipped at time $t$ is $1/|V\setminus(\cA_{t-1}\cup\cN_{t-1})|$, and that
these events occur independently.
With respect to the second factor, a similar intuition applies.
Due to the $\vecone\cbc{|\cZ_t\setminus\cZ_{t-1}|\leq\delta_t}$ factor on the left hand side,
at most $k\delta_t$ variables are added to $\cN_t$ that were not already in $\cN_{t-1}$.
Hence, for each $\PHI_{ij}$ with $\pi_{t-1}(i,j)=1$ there are now $k\delta_t$ ``good'' cases that would make $K^0_t(i,j)=1$.
Moreover, as we reveal the $k\delta_t$ variables, there remain at least $|V\setminus(\cA_{t-1}\cup\cN_{t-1})|-k\delta_t$ ``possible'' cases.
We will now establish the following.

\begin{proposition}\label{Prop_DD}
\Whp\ we have either $\cZ_{T^*}>\eps n$ or $|\cD_{T^*}|\leq2^{2-k}m$.
\end{proposition}
\begin{proof}
Let $\cE$ be the event that $\abs{\cZ_{T^*}}\leq\eps n$ but $\abs{\cD_{T^*}}>2^{2-k}m$.
Our goal is to show that $\pr\brk{\cE}=o(1)$.
To this end, we will decompose $\cE$ into various `sub-events' that are sufficiently
detailed for us to bound their probabilities via \Cor~\ref{Cor_Kt}.
In order to bound the probability of $\cE$ we will then use the union bound.

As a first step, we need to decompose $\cE$ according to the sequence $(\abs{\cZ_t\setminus\cZ_{t-1}})_{t\geq1}$
of increments of the sets $\cZ_t$.
More precisely, let $\Delta$ be the set of all sequences $\vec\delta=(\delta_t)_{1\leq t\leq T^*}$ of non-negative integers with
$\sum_{t=1}^{T^*}\delta_t\leq\eps n$.
Let $\cE(\vec\delta)$ be the event that $\abs{\cZ_t\setminus\cZ_{t-1}}\leq\delta_t$ for all $1\leq t\leq T^*$ and $\abs{\cD_{T^*}}>2^{2-k}m$.
If the event $\cE$ occurs, then there is a sequence $\vec\delta$ such that the event $\cE(\vec\delta)$ occurs.
Hence, by the union bound
	$$\pr\brk{\cE}\leq\sum_{\vec\delta\in\Delta}\pr\brk{\cE(\vec\delta)}\leq\abs\Delta\cdot\max_{\vec\delta\in\Delta}\pr\brk{\cE(\vec\delta)}.$$
As it is well known that $\abs\Delta=\bink{\eps n+T^*-1}{T^*-1}\leq\bink{\eps n+T^*}{\eps n}$, we obtain
	\begin{equation}\label{eqDD1}
	\pr\brk{\cE}\leq\bink{\eps n+T^*}{\eps n}\max_{\vec\delta\in\Delta}\pr\brk{\cE(\vec\delta)}.
	\end{equation}

Fixing any sequence $\vec\delta\in\Delta$,
	we now decompose the event $\cE(\vec\delta)$ further according to the
	precise set $M$ of clauses that end up in $\cD_{T^*}$, and according to the precise `reason' why each clause $i\in M$
	belongs to $\cD_{T^*}$.
More precisely, let $M\subset\brk m$ be a set of size $\mu=2^{2-k}m$.
Moreover, 
for disjoint $Q^*,Q^0\subset M\times\brk k$ let
$\cE_0(Q^*,Q^0)$ be the event that 
	$$\pi_0(i,j)=1\mbox{ for $(i,j)\in Q^*\cup Q^0$, while }
		\pi_0(i,j)=-1\mbox{ for $(i,j)\in M\times\brk k\setminus(Q^*\cup Q^0)$.}$$
Furthermore, for maps $\tau^*:Q^*\ra\brk{T^*}$, $\tau^0:Q^0\ra\brk{T^*}$ let $\cE(\vec\delta,\tau^*,\tau^0)$ be the event that
$\abs{\cZ_t\setminus\cZ_{t-1}}\leq\delta_t$ for all $1\leq t\leq T^*$ and
	\begin{eqnarray*}
	\pi_{\tau^*(i,j)-1}(i,j)=1&\mbox{while}&\PHI_{ij}\in\cA_{\tau^*(i,j)}\mbox{ for all }(i,j)\in Q^*,\\
	\pi_{\tau^0(i,j)-1}(i,j)=1&\mbox{while}&\PHI_{ij}\in\cN_{\tau^0(i,j)}\mbox{ for all }(i,j)\in Q^0.
	\end{eqnarray*}

If the event $\cE(\vec\delta)$ occurs, then there exist $Q^*,Q^0$ and $\tau^*,\tau^0$ such that
the events $\cE_0(Q^*,Q^0)$ and $\cE(\vec\delta,\tau^*,\tau^0)$ occur.
In fact, if $\cE(\vec\delta)$ occurs, then $\abs{\cD_{T^*}}\geq\mu$.
Thus, select a subset $M\subset \cD_{T^*}$ of size $\mu$.
By the definition of $\cD_{T^*}$, each $i\in M$ is $\cA_{T^*}\cup\cN_{T^*}$-negative.
Thus, for any $j\in\brk k$ such that $\PHI_{ij}$ is a positive literal
there is a time $1\leq t=t(i,j)\leq T^*$ such that $\pi_{t-1}(i,j)=1$ but $\pi_{t}(i,j)\in\cA_{t}\cup\cN_{t}$.
If $\pi_{t}(i,j)\in\cA_{t(i,j)}$, then include $(i,j)$ in $Q^*$ and set $\tau^*(i,j)=t$.
Otherwise, add $(i,j)$ to $Q^0$ and let $\tau^0(i,j)=t$.
Then indeed both $\cE_0(Q^*,Q^0)$ and $\cE(\vec\delta,\tau^*,\tau^0)$ occur.
Thus, by the union bound,
	\begin{equation}\label{eqDD2}
	\pr\brk{\cE(\vec\delta)}\leq\sum_{Q^*,Q^0,\tau^*,\tau^0}\pr\brk{\cE_0(Q^*,Q^0)\cap\cE(\vec\delta,\tau^*,\tau^0)}.
	\end{equation}

The event $\cE_0(Q^*,Q^0)$ depends only on the signs of the literals and is therefore $\cF_0$-measurable.
Furthermore, as signs of the literals $\PHI_{ij}$ are mutually independent, we get
	$$\pr\brk{\cE_0(Q^*,Q^0)}=2^{-k\mu}.$$
Therefore, (\ref{eqDD2}) yields	
	\begin{equation}\label{eqDD3}
	\pr\brk{\cE(\vec\delta)}\leq2^{-k\mu}\sum_{Q^*,Q^0,\tau^*,\tau^0}\pr\brk{\cE(\vec\delta,\tau^*,\tau^0)|\cF_0}.
	\end{equation}
Thus, we are left to estimate $\pr\brk{\cE(\vec\delta,\tau^*,\tau^0)|\cF_0}$.

We defined the random variables $K_t^*(\cdot,\cdot)$, $K_t^0(\cdot,\cdot)$ so that if the event $\cE(\vec\delta,\tau^*,\tau^0)$ occurs, then 
	$$\prod_{(i,j)\in Q^*}K_{\tau^*(i,j)}(i,j)\cdot\prod_{(i,j)\in Q^0}K_{\tau^0(i,j)}(i,j)\cdot\prod_{t=1}^{T_*}\vecone\cbc{|\cZ_t\setminus\cZ_{t-1}|\leq\delta_t}=1.$$
In order to apply \Cor~\ref{Cor_Kt} to the above expression, we are going to reorder the product according to the time parameter.
More precisely, let $Q_t^*=\tau^{*\,-1}(t)$ and $Q_t^0=\tau^{0\,-1}(t)$.
Then 
	\begin{eqnarray*}
	\pr\brk{\cE(\vec\delta,\tau^*,\tau^0)|\cF_0}
		&\leq&\Erw\brk{\prod_{(i,j)\in Q^*}K_{\tau^*(i,j)}(i,j)\prod_{(i,j)\in Q^0}K_{\tau^0(i,j)}(i,j)\prod_{t=1}^{T_*}\vecone\cbc{|\cZ_t-\cZ_{t-1}|\leq\delta_t}=1|\cF_0}\\
		&=&\Erw\brk{\prod_{t=1}^{T_*}K_{t}(Q_t^*)K_{t}(Q_t^0)\cdot\vecone\cbc{|\cZ_t\setminus\cZ_{t-1}|\leq\delta_t}=1|\cF_0}.
	\end{eqnarray*}
If $|\cZ_t\setminus\cZ_{t-1}|\leq\delta_t$ for all $t\leq T^*$, then $|\cN_{t-1}|+k\delta_t\leq k\sum_{s\leq t}\delta_t\leq k\eps n$ for all $t\leq T^*$.
Furthermore, $\abs{\cA_t}\leq t\leq T^*=\frac{n}{k}$ for all $t\geq0$.
Hence, $|V\setminus(\cA_{t-1}\cup\cN_{t-1})|-k\delta_t\geq n(1-k\eps-1/k)\geq n/1.01$ for all $t\leq T^*$,
provided that $k\geq k_0$ is large enough.
Thus, \Cor~\ref{Cor_Kt} entails in combination with \Lem~\ref{Lemma_filt}
	\begin{eqnarray}
	\pr\brk{\cE(\vec\delta,\tau^*,\tau^0)|\cF_0}
		&=&\bcfr{1.01}{n}^{|Q^*|}\cdot\prod_{(i,j)\in Q^0}\frac{1.01k\delta_{\tau^0(i,j)}}{n}.\label{eqDD4}
	\end{eqnarray}

For any $M\subset\brk m$ of size $\mu$ and any two disjoint $Q^*,Q^0\subset M\times\brk k$ let
	$$S(M,Q^*,Q^0)=\sum_{\tau^*,\tau^0}\bcfr{1.01}{n}^{|Q^*|}\cdot\prod_{(i,j)\in Q^0}\frac{1.01k\delta_{\tau^0(i,j)}}{n},$$
with the sum ranging over all maps $\tau^{*}:Q^{*}\ra\brk{T^*}$, $\tau^{0}:Q^{0}\ra\brk{T^*}$.
Recall that $\theta=T^*/n$.
As $\sum_{t\leq T^*}\delta_t\leq\eps n$, we obtain
	\begin{eqnarray}\nonumber
	S(M,Q^*,Q^0)&\leq&\bcfr{1.01 T^*}{n}^{|Q^*|}\bcfr{1.01k}{n}^{|Q^0|}\sum_{\tau^0}\prod_{(i,j)\in Q^0}\delta_{\tau^0(i,j)}\\
		&=&\bcfr{1.01 T^*}{n}^{|Q^*|}\bcfr{1.01k}{n}^{|Q^0|}\bc{\sum_{t=1}^{T^*}\delta_t}^{|Q^0|}
		\leq\bc{1.01\theta}^{|Q^*|}\bc{1.01\eps k}^{|Q^0|}.
	\label{eqDD5}
	\end{eqnarray}
Combining (\ref{eqDD3}), (\ref{eqDD4}), and~(\ref{eqDD5}), we thus get for any $\vec\delta\in\Delta$
	\begin{eqnarray}
	\pr\brk{\cE(\vec\delta)}&\leq&2^{-k\mu}
		\sum_{M\subset\brk m:|M|=\mu}\sum_{Q^*,Q^0\subset M\times\brk k:Q^*\cap Q^0=\emptyset}S(M,Q^*,Q^0)\nonumber\\
	&\leq&2^{-k\mu}\bink{m}\mu
		\sum_{q^*,q^0:q^*+q^0\leq k\mu}\sum_{Q^*,Q^0:|Q^*|=q^*,|Q^0|=q^0}
			\bc{1.01\theta}^{q^*}\bc{1.01\eps k}^{q^0}\nonumber\\
	&\leq&2^{-k\mu}
		\bink{m}\mu\sum_{q^*,q^0:q^*+q^0\leq k\mu}
		\bink{k\mu}{q^*,q^0,k\mu-q^*-q^0}(1.01\theta)^{q^*}\bc{1.01k\eps}^{q^0}\nonumber\\
	&\leq&\bink{m}\mu\bcfr{1+1.01(\theta+k\eps)}2^{k\mu}\nonumber
		\leq\brk{\frac{\eul m}\mu\cdot \bcfr{1+1.01(\theta+k\eps)}2^{k}}^\mu\nonumber\\
		&\leq&\brk{\eul 2^{k-2}\cdot \bcfr{1+1.01(\theta+k\eps)}2^{k}}^\mu\leq0.999^\mu,
			\label{eqDD6}
	\end{eqnarray}
provided that $k\geq k_0$ is sufficiently big. 
Finally, combining~(\ref{eqDD1}) and~(\ref{eqDD6}), we obtain
	\begin{eqnarray}
	\pr\brk{\cE}&\leq&\bink{\eps n+T^*}{\eps n}0.999^\mu
		\leq\bcfr{\eul(\eps n+\theta n)}{\eps n}^{\eps n}0.999^\mu
		\leq\bc{\eul(1+\theta/\eps)}^{\eps n}0.999^\mu.
			\label{eqzw}
	\end{eqnarray}
By our assumption that $\rho\geq k^{-3}$ (cf.\ the first paragraph in \Sec~\ref{Sec_Outline}),
	we have $\mu=2^{2-k}m\geq\rho n/k\geq k^{-4}n$.
Hence, recalling that $\theta\leq1/k$ and $\eps=\exp(-k^{2/3})$ (cf.~(\ref{eqlambdaeps})),
	we obtain from~(\ref{eqzw})
	\begin{eqnarray*}
	\pr\brk{\cE}&\leq&\exp\brk{n\bc{\eps\ln(2\eul/\eps)-k^{-4}}}
		\leq\exp\brk{n\bc{k\exp(-k^{2/3})+k^{-4}\ln0.999}}=\exp(-\Omega(n))=o(1),
	\end{eqnarray*}
provided that $k\geq k_0$ is sufficiently large.
\qed\end{proof}

\noindent
Finally, \Prop~\ref{Prop_D} is immediate from \Prop s~\ref{Prop_Z} and~\ref{Prop_DD}.

\section{Proof of \Prop~\ref{Prop_Z}}\label{Sec_Z}

{\em Throughout this section we keep the notation and the assumptions of \Prop~\ref{Prop_Z}.}

\subsection{Outline}

The goal in this section is to bound the size of the set $\cZ_{T^*}$.
There are two reasons why step {\bf PI2} may add a clause index $i\in\brk m$ to the set $\cZ_{t}$ for some $1\leq t\leq T^*$.
First, the clause $\PHI_i$ may feature at least $k_1$ variables from the set $\cA_{t-1}\cup\cbc{|\PHI_{i_tj_t}|}$,
i.e., variables that have been flipped at least once.
Second, $\PHI_i$ may contain at least $\lambda$ variables that also occur in clauses that were added to $\cZ_{t}$ previously.
The key issue is to deal with the first case.
Once that is done, we can bound the number of clauses that get included for the second reason
via \Lem~\ref{Lemma_core}, i.e., via the expansion properties of the random formula.

Thus, we need to investigate how a clause $\PHI_i$ comes to contain a lot of variables from $\cA_{t-1}\cup\cbc{|\PHI_{i_tj_t}|}$
for some $t\leq T^*$.
There are two ways in which this may occur.
First, \Walksat\ may have tried to satisfy $\PHI_i$ `actively' several times, i.e., $i_s=i$ for several $s\leq t$.
Second, $\PHI_i$ may contain several of the variables $|\PHI_{i_sj_s}|$ flipped at times $s<t$
`accidentally', i.e., without \Walksat\ trying to actively satisfy $i$.
More precisely, for any $t\geq0$ we call a pair $(i,j)\in\brk m\times\brk k$ 
\begin{enumerate}
\item[$\bullet$] \emph{$t$-active} if there is $1\leq s\leq t$ such that
	$(i,j)=(i_s,j_s)$ and $\pi_{s-1}(i,j)=-1$.
\item[$\bullet$] \emph{$t$-passive} if there is $1\leq s\leq t$ such that $(i,j)\neq(i_s,j_s)$
	but $\abs{\PHI_{ij}}=\abs{\PHI_{i_sj_s}}$ and $\pi_{s-1}(i,j)\in\cbc{-1,1}$.
\end{enumerate}
Furthermore, we say that $i\in\brk m$ is \emph{$t$-active} if there are $k_2=k_1-10^{-6}k$ indices $j$ such that $(i,j)$ is $t$-active.
Similarly, we say that $i$ is \emph{$t$-passive} if there are $k_3=10^{-6}k$ indices $j$ such that $(i,j)$ is $t$-passive.
These definitions ensure that any $i\in\brk m$ for which there are at least $k_1$ indices $j\in\brk k$ such that
$|\PHI_{ij}|\in\cA_{t-1}\cup\cbc{|\PHI_{i_tj_t}|}$ is either $t$-active or $t$-passive.

To prove \Prop~\ref{Prop_Z}, we will deal separately with $t$-active and $t$-passive clauses.
Let $A_t$ be the number of $t$-active clauses, and let $P_t$ be the number of $t$-passive clauses.

\begin{lemma}\label{Lemma_active}
For any $1\leq t\leq T^*$  we have
	$\pr\brk{A_t< \eps n/4\vee|\cZ_{t}|>\eps n}\geq1-1/n^2.$
\end{lemma}
We defer the proof of \Lem~\ref{Lemma_active} to \Sec~\ref{Sec_active}.

\begin{lemma}\label{Lemma_passive}
For any $1\leq t\leq T^*$  we have
$\pr\brk{P_t< \eps n/4\vee|\cZ_{t-1}|>\eps n
		}\geq1-1/n^2$.
\end{lemma}
\begin{proof}
As in the proof of \Prop~\ref{Prop_DD}, 
we are going to break the event of interest, i.e.,
	$$\cE=\cbc{P_t\geq\eps n/4\wedge|\cZ_{t-1}|\leq\eps n},$$ down into sub-events
whose probabilities can be estimated via \Lem~\ref{Lemma_Kt*}.
Then we will use the union bound to estimate the probability of $\cE$.

For a set $M\subset\brk m$ of $\mu=\eps n/4$ clause indices let $\cE(M)$ be the event that $|\cZ_{t-1}|\leq\eps n$ and all $i\in M$ are $t$-passive.
If $\cE$ occurs, then there is a set $M$ such that the event $\cE(M)$ occurs.
Hence, by the union bound
	\begin{eqnarray}\label{eqpassive1}
	\pr\brk{\cE}&\leq&\sum_{M\subset\brk m:|M|=\mu}\pr\brk{\cE(M)}\leq\bink m\mu\max_M\pr\brk{\cE(M)}.
	\end{eqnarray}

Thus, fix a set $M\subset\brk m$ of size $\mu$.
Let $Q\subset M\times\brk k$ be a set such that for each $i\in M$ there are precisely $k_3$ indices $j\in\brk k$ such that $(i,j)\in Q$.
Let $\cE(M,Q)$ be the event that $|\cZ_{t-1}|\leq\eps n$ and
all pairs $(i,j)\in Q$ are $t$-passive.
If the event $\cE(M)$ occurs, then there exists a set $Q$ such that $\cE(M,Q)$ occurs.
Therefore, again by the union bound
	\begin{eqnarray}\label{eqpassive2}
	\pr\brk{\cE(M)}&\leq&\sum_{Q}\pr\brk{\cE(M,Q)}\leq\bink k{k_3}^\mu\max_Q\pr\brk{\cE(M,Q)}.
	\end{eqnarray}

For a map $\tau:Q\ra\brk t$ let $\cE(M,Q,\tau)$ be the event that $|\cZ_{t-1}|\leq\eps n$ and
	$$\tau(i,j)=\min\cbc{s\in\brk t:(i,j)\mbox{ is $s$-passive}}\mbox{ for all }(i,j)\in Q.$$
If the event $\cE(M,Q)$ occurs, then there is a map $\tau$ such that the event $\cE(M,Q,\tau)$ occurs.
Consequently, for any $M,Q$ we have
	\begin{eqnarray}\label{eqpassive3}
	\pr\brk{\cE(M,Q)}&\leq&\sum_{\tau}\pr\brk{\cE(M,Q,\tau)}\leq t^{|Q|}\max_{\tau}\pr\brk{\cE(M,Q,\tau)}.
	\end{eqnarray}
Combining~(\ref{eqpassive1}), (\ref{eqpassive2}), and~(\ref{eqpassive3}), we see that
	\begin{eqnarray}\label{eqpassive4}
	\pr\brk{\cE}&\leq&\bink m\mu\bink k{k_3}^\mu t^{k_3\mu}\max_{M,Q,\tau}\pr\brk{\cE(M,Q,\tau)}.
	\end{eqnarray}

Hence, fix any $M,Q,\tau$.
Let $Q_s=\tau^{-1}(s)$ for any $1\leq s\leq t$,
	and let $\cE_s^*(Q_s)$ be the event that
		$|\PHI_{ij}|=|\PHI_{i_tj_t}|\not\in\cA_{t-1}\cup\cN_{t-1}$ for all $(i,j)\in Q_s$, and $(i_t,j_t)\not\in Q_s$.
If $\cE(M,Q,\tau)$ occurs, then the events $\cE_s^*(Q_s)$ occur for all $1\leq s\leq t$.
Moreover, the construction {\bf PI0--PI3} ensures that $|\cA_s|\leq s$, and that
	 $|\cN_{s-1}|\leq k|\cZ_{s-1}|\leq k\eps n$ for all $1\leq s\leq t$.
Therefore, \Lem~\ref{Lemma_Kt*} implies
	\begin{eqnarray}
	\pr\brk{\cE(M,Q,\tau)}&\leq&\pr\brk{\bigcap_{s=1}^t\cE_s^*(Q_s)\cap\cbc{|\cN_{s-1}|\leq k\eps n}}
		\leq\prod_{s=1}^t\max\cbc{1,n-s+1-k\eps n}^{-|Q_s|}.
			\label{eqpassive5a}
	\end{eqnarray}
As $s\leq t\leq T^*\leq n/k$, $\eps=\exp(-k^{2/3})$, and because we are assuming that $k\geq k_0$ is sufficiently large,
we have $n-s+1-k\eps n\geq n/1.001$.
Hence, (\ref{eqpassive5a}) yields
	\begin{eqnarray}
	\pr\brk{\cE(M,Q,\tau)}&\leq&
		\prod_{s=1}^t\max\cbc{1,n-s+1-k\eps n}^{-|Q_s|}\leq(1.001/n)^{\mu k_3}.
			\label{eqpassive5}
	\end{eqnarray}
Finally, combining~(\ref{eqpassive4}) and~(\ref{eqpassive5}) and recalling that $\theta=T^*/n$, we get
	\begin{eqnarray*}
	\pr\brk{\cE}&\leq&\bink m\mu\bink k{k_3}^\mu t^{k_3\mu}(1.001/n)^{\mu k_3}
			\leq\brk{\frac{\eul m}{\mu}\cdot\bcfr{1.001\eul k\theta}{k_3}^{k_3}}^\mu
				\leq\brk{\frac{4\eul 2^k\rho}{\eps k}\bcfr{1.001\eul k\theta}{k_3}^{k_3}}^\mu.
	\end{eqnarray*}
By our choice of $\theta$ we have $1.001\eul k\theta\leq10$.
Hence, we obtain for $k\geq k_0$ large enough
	\begin{eqnarray*}
	\pr\brk{\cE}&\leq&
				\brk{\frac{4\eul 2^k\rho}{\eps k}k_3^{-k_3/2}}^\mu\leq\exp(-\mu)=o(1),
	\end{eqnarray*}
thereby completing the proof.
\qed\end{proof}

\noindent\emph{Proof of \Prop~\ref{Prop_Z}.}
In order to bound $|\cZ_t|$ for $0\leq t\leq T^*$, we are going to consider a superset $\cY_t\supset\cZ_t$ whose
size is easier to estimate.
To define $\cY_t$, we let $\cY_t^*$ be the set of all $i$ that are either $t$-active or $t$-passive.
Now, $\cY_t$ is the outcome of the following process.
\begin{quote}
Initially, let $\cY_t=\cY_t^*$.\\
While there is a clause $i\in\brk m\setminus\cY_t$ such that
	$\abs{\cbc{j\in\brk k:|\PHI_{ij}|\in N(\PHI_{\cY_t})}}\geq\lambda,$
	add $i$ to $\cY_t$.
\end{quote}
Comparing the above process with the construction in {\bf PI2}, we see that indeed
	\begin{equation}\label{eqYZ}
	\cY_t\supset\cZ_t.
	\end{equation}
Also note that $\cY_t\supset\cY_{t-1}$ for all $t\geq1$.

To bound $|\cY_t|$, we proceed by induction on $t$.
Let $Y_t$ be the event that either the random formula $\PHI$ violates the property~(\ref{eqcore}), or
	$\abs{\cY_t}>\eps n$.
We claim that $\pr\brk{Y_0}=o(1)$ and that
	\begin{equation}\label{eqyt}
	\pr\brk{Y_t}\leq\pr\brk{Y_{t-1}}+2n^{-2}\qquad\mbox{for all }1\leq t\leq T^*.
	\end{equation}

Since trivially $\cY_0=\emptyset$, $Y_0$ is simply the event that $\PHI$ violates~(\ref{eqcore}).
Hence, \Lem~\ref{Lemma_core}  shows directly that
	\begin{equation}\label{eqY0}
	\pr\brk{Y_0}=o(1).
	\end{equation}
Now, consider some $1\leq t\leq T^*$.
 \Lem s~\ref{Lemma_active} (applied to $t-1$) and \Lem~\ref{Lemma_passive} (applied to $t$) show that
	$$\pr\brk{A_t+P_t\leq\eps n/2\vee|\cZ_{t-1}|>\eps n}\geq1-2/n^2.$$
Furthermore, if $Y_{t-1}$ does not occur, then we know that $|\cZ_{t-1}|\leq|\cY_{t-1}|\leq\eps n$
and that~(\ref{eqcore}) is satisfied.
If in addition $A_t+P_t\leq\eps n/2$, then (\ref{eqcore})  ensures that
$|\cY_t|\leq\eps n$, and thus $Y_t$ does not occur.
Therefore,

	\begin{eqnarray*}
	\pr\brk{Y_t}&=&\pr\brk{Y_{t-1}}+\pr\brk{Y_t\setminus Y_{t-1}}
		\leq\pr\brk{Y_{t-1}}+\pr\brk{A_t+P_t>\frac{\eps n}2\wedge|\cZ_{t-1}|\leq\eps n}\leq \pr\brk{Y_{t-1}}+2/n^{2}.
	\end{eqnarray*}

Finally, (\ref{eqyt}) and~(\ref{eqY0}) yield
	$$\pr\brk{|\cY_{T^*}|>\eps n}\leq\pr\brk{Y_{T^*}}\leq\pr\brk{Y_0}+\sum_{t=1}^{T^*}2/n^2= o(1)+2T^*/n^2=o(1).$$
In combination with~(\ref{eqYZ}), this implies the assertion.
\qed

\subsection{Proof of \Lem~\ref{Lemma_active}}\label{Sec_active}

How can a clause $\PHI_i$ become $t$-active?
If this occurs, then \Walksat\ must have tried `actively' to satisfy $\PHI_i$ at least $k_2$ times by flipping one of its variables.
But each time, the variable that \Walksat\ flipped to satisfy $\PHI_i$ got flipped again because flipping it rendered another clause unsatisfied.

More precisely, if $\PHI_i$ is $t$-active, then there exist distinct `slots' $j_1,\ldots,j_{k_2}\in\brk k$ and times $s_1,\ldots,s_{k_2}\in\brk t$
such that $(i,j_l)$ is $s_l$-active for $l=1,\ldots,k_2$.
This means that at the times $s_l$, \Walksat\ actively tried to satisfy $\PHI_i$ by flipping $|\PHI_{i j_l}|$ ($l=1,\ldots,k_2$).
However, as \Walksat\ had to make $k_2$ attempts, each of the variables $|\PHI_{i j_l}|$ with $l<k_2$ must have been flipped
once more by time $s_{l+1}$.
Hence, $|\PHI_{i j_l}|$ occurs positively in a clause $\PHI_{h_l}$ that is unsatisfied at some time $s_l<q_l<s_{l+1}$.
In particular, $h_l\in\cD_{q_l}\subset\cD_t$.

Thus, in order to prove \Lem~\ref{Lemma_active} we are going to bound the probability that there are at least $\eps n/4$
clauses $\PHI_i$ that admit $j_1,\ldots,j_{k_2}\in\brk k$ 
such that for each $1\leq l<k_2$ there is another clause $\PHI_{h_l}$ with the following properties.
\begin{description}
\item[A1.] We have $\sign(\PHI_{i j_l})=-1$,
		and there is an index $j\in\brk k$ such that $\sign(\PHI_{h_l j})=1$ and $\PHI_{h_lj}=|\PHI_{i j_l}|$.
\item[A2.] $h_l\in\cD_t$, i.e., $\PHI_{h_l}$ is $\cA_t\cup\cN_t$-negative.
\end{description}	

In order to deal with {\bf A1} we will need to refine our filtration.
Given a subset $Q\subset\brk m\times\brk k$ and a map $g:Q\ra\brk m\times\brk k$, we let
$\Omega_g$ be the event that 
	$$\sign(\PHI_{ij})=-1,\
		\sign(\PHI_{g(i,j)})=1\mbox{ and }|\PHI_{ij}|=\abs{\PHI_{g(i,j)}}\mbox{ for all }(i,j)\in Q.$$
Since the literals of the random formula $\PHI$ are independently uniformly distributed, we see that
	\begin{equation}\label{eqOmegag}
	\pr\brk{\Omega_g}\leq2^{-|Q\cup g(Q)|}n^{-|Q|}.
	\end{equation}
We consider $\Omega_g$ as a probability space equipped with the uniform distribution
	(in other words, we are going to condition on $\Omega_g$).
Further, we define a filtration $(\cF_{g,t})_{t\geq0}$ on $\Omega_g$ 
by letting $\cF_{g,t}=\cbc{\cE\cap\Omega_g:\cE\in\cF_t}$.
In other words, $\cF_{g,t}$ is the projection of $\cF_t$ onto $\Omega_g$.
Hence, Fact~\ref{Fact_measurable} directly implies the following.

\begin{fact}\label{Fact_gmeasurable}
For any $t\geq0$, any $x\in V$, and any $i\in\brk m$ the events $\cbc{\sigma_t(x)=1}$,
	$\cbc{\mbox{$\PHI_i$ is satisfied under $\sigma_t$}}$,
		$\cbc{x\in\cA_t}$, $\cbc{i\in\cZ_t}$, $\cbc{x\in\cN_t}$, and $\cbc{T=t}$ are $\cF_{g,t}$-measurable.
\end{fact}

Moreover, since the only conditioning we impose in $\Omega_g$ concerns the literals $\PHI_{ij}$ with $(i,j)\in Q\cup g(Q)$,
Fact~\ref{Fact_iid} yields the following.

\begin{fact}\label{Fact_giid}
Let $t\geq0$.
Let $\cE_t$ be the set of all pairs $(i,j)\in\brk m\times\brk k\setminus(Q\cup g(Q))$ such that $\pi_t(i,j)\in\{-1,1\}$.
The conditional joint distribution of the variables $(|\PHI_{ij}|)_{(i,j)\in\cE_t}$ given $\cF_{t,g}$ is uniform over $(V\setminus(\cA_t\cup\cN_t))^{\cE_t}$.
That is, 
for any map $f:\cE_t\rightarrow V\setminus (\cA_t\cup\cN_t)$ we have
	$$\pr\brk{\forall (i,j)\in\cE_t:|\PHI_{ij}|=f(i,j)|\cF_{t,g}}=|V\setminus(\cA_t\cup\cN_t)|^{-|\cE_t|}.$$
\end{fact}

Similarly, with respect to the random variables $K_t^*(\cdot,\cdot)$ and $K_t^0(\cdot,\cdot)$ defined in~(\ref{eqKt*}) and (\ref{eqKt0}),
\Cor~\ref{Cor_Kt} implies the following.

\begin{corollary}\label{Cor_gK}
For any $t\geq1$, $\delta_t\geq0$ and $\cI^*,\cI^0\subset\brk m\times\brk k\setminus(Q\cup g(Q))$ we have
	\begin{eqnarray*}
	\Erw\brk{K_t^*(\cI^*)K_t^0(\cI^0)\vecone\cbc{|\cZ_t\setminus\cZ_{t-1}|\leq\delta_t}|\cF_{g,t-1}}&\leq&
		\max\cbc{1,|V\setminus(\cA_{t-1}\cup\cN_{t-1})|}^{-\abs{\cI^*}}\\
		&&\quad\cdot\bcfr{k\delta_t}{\max\cbc{1,|V\setminus(\cA_{t-1}\cup\cN_{t-1})|-k\delta_t}}^{\abs{\cI^0}}.
	\end{eqnarray*}
\end{corollary}

As a further preparation, we need the following lemma.

\begin{lemma}\label{Lemma_complicated}
Let $1\leq t\leq T^*$.
Set $\mu=\eps n/4$ and let $M\subset\brk m$ be a set of size $|M|=\mu$.
Furthermore, let $Q\subset M\times\brk k$,
let $I\subset\brk m$ be a set of size $|I|\leq|Q|$, 
and let $g:Q\ra I\times\brk k$.
Let $\cE(M,Q,I,g)$ denote the event that $|\cZ_{t}|\leq\eps n$
and the following three statements hold.
\begin{enumerate}
\item[a.] For all $(i,j)\in Q$ we have $\sign(\PHI_{ij})=-1$, $\sign(\PHI_{g(i,j)})=1$, and $\PHI_{g(i,j)}=|\PHI_{ij}|$. 
\item[b.] $I\subset\cD_t$.
\item[c.] For each $i\in I$ there is $j\in\brk k$ such that $(i,j)\in g(Q)$.
\end{enumerate}
Then
	$\pr\brk{\cE(M,Q,I,g)}\leq	2\bink{T^*+\eps n}{\eps n}(2n)^{-|Q|}2^{-|I\times\brk k|}
			\exp(1.011k\theta|I|).$
\end{lemma}
\begin{proof}
To estimate $\pr\brk{\cE(M,Q,I,g)}$, we need to decompose the event $\cE(M,Q,I,g)$ into
`more detailed' sub-events whose probabilities can be bounded directly via \Cor~\ref{Cor_gK}.
To this end, let $\cI^*,\cI^0$ be two disjoint subsets of $I\times\brk k\setminus g(Q)$, and let
$t^*:\cI^*\ra\brk{T^*}$, $t^0:\cI^0\ra\brk{T^*}$ be two maps.
Let $\cE(M,Q,I,g,t^*,t^0)$ be the event that $|\cZ_{t}|\leq\eps n$ and that the following statements are true.
\begin{enumerate}
\item[a.] For all $(i,j)\in Q$ we have $\sign(\PHI_{ij})=-1$, $\sign(\PHI_{g(i,j)})=1$, and $\PHI_{g(i,j)}=|\PHI_{ij}|$. 
\item[b.]
	\begin{enumerate}
	\item[\em i.] If $(i,j)\in I\times\brk k\setminus(g(Q)\cup\cI^*\cup\cI^0)$, then $\sign(\PHI_{ij})=-1$.
	\item[\em ii.] If $(i,j)\in\cI^*$, then $\sign(\PHI_{ij})=\pi_{t^*(i,j)-1}(i,j)=1$ and $\PHI_{ij}\in\cA_{t^*(i,j)}$.
	\item[\em iii.] If $(i,j)\in\cI^0$, then $\sign(\PHI_{ij})=\pi_{t^0(i,j)-1}(i,j)=1$ and $\PHI_{ij}\in\cN_{t^0(i,j)}$.
	\end{enumerate}
\item[c.] For each $i\in I$ there is $j\in\brk k$ such that $(i,j)\in g(Q)$.
\end{enumerate}
If the event $\cE(M,Q,I,g)$ occurs, then there exist $\cI^*,\cI^0,t^*,t^0$ such that the event $\cE(M,Q,I,g,t^*,t^0)$ occurs.
Indeed, the definition of the set $\cD_{T^*}$ is such that if $i\in\cD_{T^*}$, then for any $(i,j)\in I\times\brk k$ such that
$\sign(\PHI_{ij})=1$ we have $\PHI_{ij}\in\cA_{T^*}\cup\cN_{T^*}$.
Thus, by the union bound,
	\begin{eqnarray}\label{eqactive4}
	\pr\brk{\cE(M,Q,I,g)}&\leq&\sum_{\cI^*,\cI^0}\sum_{t^*,t^0}\pr\brk{\cE(M,Q,I,g,t^*,t^0)}.
	\end{eqnarray}

Furthermore, let $\vec\delta=(\delta_1,\ldots,\delta_{t})$ be a sequence such that $\sum_{s=1}^{t}\delta_s\leq\eps n$.
Let $\cE(\vec\delta,M,Q,I,g,t^*,t^0)$ be the event that $\abs{\cZ_s\setminus\cZ_{s-1}}\leq\delta_s$ for all $1\leq s<t$
and that $\cE(M,Q,I,g,t^*,t^0)$ occurs.
Then by the union bound,	
	\begin{eqnarray}\nonumber
	\pr\brk{\cE(M,Q,I,g,t^*,t^0)}&\leq&\sum_{\vec\delta}\pr\brk{\cE(\vec\delta,M,Q,I,g,t^*,t^0)}\\
			&\leq&\bink{T^*+\eps n}{\eps n}\max_{\vec\delta}\pr\brk{\cE(\vec\delta,M,Q,I,g,t^*,t^0)}.
				\label{eqactive5}
	\end{eqnarray}
	
The event $\cE(\vec\delta,M,Q,I,g,t^*,t^0)$ is sufficiently specific so that we can estimate its probability easily.
Namely, if $\cE(\vec\delta,M,Q,I,g,t^*,t^0)$ occurs, then $\Omega_g$ occurs and
	\begin{equation}\label{eqcomplicated1}
	\prod_{(i,j)\in\cI^*}K^*_{t^*(i,j)}(i,j)\prod_{(i,j)\in\cI^0}K^0_{t^*(i,j)}(i,j)\prod_{s=1}^{t}\vecone\cbc{\abs{\cZ_{s}\setminus\cZ_{s-1}}\leq\delta_s}=1.
	\end{equation}
To bound the probability that~(\ref{eqcomplicated1}) occurs, we reorder the product by the time parameter.
That is, letting $\cI^*_s=t^{*\,-1}(s)$, $\cI^0_s=t^{0\,-1}(s)$, we get
	\begin{eqnarray}\nonumber
	\pr\brk{\cE(\vec\delta,M,Q,I,g,t^*,t^0)|\cF_{g,0}}\\
		&\hspace{-6cm}\leq&\hspace{-3cm}\;
			\Erw\brk{\prod_{(i,j)\in\cI^*}K_{t^*(i,j)}(i,j)\prod_{(i,j)\in\cI^0}K_{t^*(i,j)}(i,j)\prod_{s=1}^{t}\vecone\cbc{\abs{\cZ_{s}\setminus\cZ_{s-1}}\leq\delta_s}=1|\cF_{g,0}}
				\nonumber\\
		&\hspace{-6cm}\leq&\hspace{-3cm}\;
			\Erw\brk{\prod_{s=1}^tK_s^*(\cI^*_s)K_s^0(\cI^0_s)\vecone\cbc{\abs{\cZ_{s}\setminus\cZ_{s-1}}\leq\delta_s}|\cF_{g,0}}.
								\label{eqactive6}
	\end{eqnarray}
Since for any $s\leq t\leq T^*$ we have $\abs{\cA_s}\leq s\leq T^*\leq\frac{n}{k}$, and as $\abs{\cN_s}\leq k\sum_{q=1}^s\delta_s\leq k\eps n$,
we see that $\abs{\cA_s\cup\cN_s}+k\delta_s\leq 0.001n$ for all $s\leq t$.
Hence, (\ref{eqactive6}) and~\Cor~\ref{Cor_gK} yield
	\begin{eqnarray}\nonumber
	\pr\brk{\cE(\vec\delta,M,Q,I,g,t^*,t^0)|\cF_{g,0}}&\leq&\prod_{s=1}^t\bcfr{1.01}{n}^{|\cI^*_s|}\bcfr{1.01k\delta_s}{n}^{|\cI^0_s|}\\
		&\leq&\bcfr{1.01}{n}^{|\cI^*|+|\cI^0|}\prod_{(i,j)\in\cI^0}k\delta_{t^0(i,j)}.
			\label{eqactive7}
	\end{eqnarray}
Furthermore, if the event $\cE(\vec\delta,M,Q,I,g,t^*,t^0)$ occurs, then for all $(i,j)\in I\times\brk k\setminus(g(Q)\cup\cI^*\cup\cI^0)$
we have $\sign(\PHI_{ij})=-1$, while $\sign(\PHI_{ij})=1$ for all $(i,j)\in\cI^*\cup\cI^0$.
This event is $\cF_{0,g}$-measurable.
Hence, as the signs of the literals $\PHI_{ij}$ are independently uniformly distributed, we obtain from~(\ref{eqactive7})
	\begin{eqnarray}
	\pr\brk{\cE(\vec\delta,M,Q,I,g,t^*,t^0)|\Omega_g}&\leq&
			2^{-|I\times\brk k\setminus g(Q)|}\bcfr{1.01}{n}^{|\cI^*\cup\cI^0|}\prod_{(i,j)\in\cI^0}k\delta_{t^0(i,j)}.
							\label{eqactive8}
	\end{eqnarray}
Combining~(\ref{eqOmegag}) and~(\ref{eqactive8}), we get
	\begin{eqnarray}\nonumber
	\pr\brk{\cE(\vec\delta,M,Q,I,g,t^*,t^0)}&=&
		\pr\brk{\Omega_g}\pr\brk{\cE(\vec\delta,M,Q,I,g,t^*,t^0)|\Omega_g}\\
		&\leq&(2n)^{-|Q|}
			2^{-|I\times\brk k|}\bcfr{1.01}{n}^{|\cI^*\cup\cI^0|}\prod_{(i,j)\in\cI^0}k\delta_{t^0(i,j)}.
							\label{eqactive9}
	\end{eqnarray}

As~(\ref{eqactive4}) and~(\ref{eqactive5}) show, in order to obtain $\pr\brk{\cE(M,Q,I,g)}$, we need to sum~(\ref{eqactive9}) over all possible choices of
$\vec\delta,\cI^*,\cI^0,t^*,t^0$:
	\begin{eqnarray*}
	\pr\brk{\cE(M,Q,I,g)}&\leq&
			\bink{T^*+\eps n}{\eps n}(2n)^{-|Q|}\sum_{\cI^*,\cI^0}
				\sum_{t^*:\cI^*\ra\brk t}\sum_{t^0:\cI^0\ra\brk t}\pr\brk{\cE(\vec\delta,M,Q,I,g,t^*,t^0)}\\
	&\leq&\bink{T^*+\eps n}{\eps n}(2n)^{-|Q|}
		2^{-|I\times\brk k|}
		\sum_{\cI^*,\cI^0}\bcfr{1.01}{n}^{|\cI^*\cup\cI^0|}t^{|\cI^*|}\sum_{t^0:\cI^0\ra\brk t}
			\prod_{(i,j)\in\cI^0}k\delta_{t^0(i,j)}\\
	&\leq&\bink{T^*+\eps n}{\eps n}(2n)^{-|Q|}
		2^{-|I\times\brk k|}
			\sum_{\cI^*,\cI^0}\bcfr{1.01}{n}^{|\cI^*\cup\cI^0|}t^{|\cI^*|}\bc{\sum_{s=1}^tk\delta_s}^{|\cI^0|}\\
	&\leq&\bink{T^*+\eps n}{\eps n}(2n)^{-|Q|}2^{-|I\times\brk k|}
			\sum_{\cI^*,\cI^0}
				\bc{1.01 t/n}^{|\cI^*|}\bc{1.01k\eps}^{|\cI^0|}\quad[\mbox{as $\sum_{s=1}^t\delta_s\leq\eps n$}]\\
	&\leq&\bink{T^*+\eps n}{\eps n}(2n)^{-|Q|}2^{-|I\times\brk k|}
		(1+1.01(\theta+k\eps))^{k |I|}\\
	\end{eqnarray*}
Hence,
	\begin{eqnarray*}\label{eqactive10}
	\pr\brk{\cE(M,Q,I,g)}&\leq&2\bink{T^*+\eps n}{\eps n}(2n)^{-|Q|}2^{-|I\times\brk k|}
			\exp(1.011k|I|\theta).
	\end{eqnarray*}
	as desired.
\qed\end{proof}

\noindent\emph{Proof of \Lem~\ref{Lemma_active}.}
Let $\mu=\eps n/4$ and fix some $1\leq t\leq T^*$.
Let $\cE$ be the event that $|\cZ_{t}|\leq\eps n$ and $A_t\geq\mu$.
For a set $M\subset\brk m$ of size $|M|=\mu$ we let $\cE(M)$ signify the event that all clauses $i\in M$ are $t$-active.
If $\cE$ occurs, then there is a set $M$ of size $\mu$ such that $\cE(M)$ occurs.
Hence, by the union bound
	\begin{eqnarray}\label{eqactive1}
	\pr\brk{\cE}&\leq&\sum_{M\subset\brk m:|M|=\mu}\pr\brk{\cE(M)}\leq\bink{m}\mu\max_{M}\pr\brk{\cE(M)}.
	\end{eqnarray}
	
To bound the expression on the r.h.s., fix some set $M\subset\brk m$ of size $\mu$.
Let $\cQ(M)$ be the set of all $Q\subset M\times\brk k$ such that for each $i\in M$ we have
	$|\cbc{j\in\brk k:(i,j)\in Q}|= k_2-1$.
For a set $Q\in\cQ(M)$ let $\cE(M,Q)$ be the event that $|\cZ_{t}|\leq\eps n$ and
\begin{enumerate}
\item[a.] all pairs $(i,j)\in Q$ are $s(i,j)$-active for some $s(i,j)\leq t$, and
\item[b.] for each $i\in M$ there is $j'\in\brk k$ such that $(i,j')$ is $s$-active at some time $s$ satisfying
	 $$\max_{j:(i,j)\in Q}s(i,j)<s\leq t.$$
\end{enumerate}
If the event $\cE(M)$ occurs, then there exists $Q\in\cQ(M)$ such that $\cE(M,Q)$ occurs.
(In fact, if $\cE(M)$ occurs, then by the definition of $t$-active, for any $i\in M$ there are at least $k_2$ indices $j$ such that
$(i,j)$ is $s$-active for some $s\leq t$.
We can thus let $Q$ contain the pairs $(i,j)$ for the `earliest' $k_2-1$ such indices $j$.)
Hence, by the union bound
	\begin{eqnarray}\label{eqactive2}
	\pr\brk{\cE(M)}&\leq&\sum_{Q\in\cQ}\pr\brk{\cE(M,Q)}\leq\bink{k}{k_2-1}^{\mu}\max_{Q\in\cQ}\pr\brk{\cE(M,Q)}.
	\end{eqnarray}

Now, fix a set $M\subset\brk m$, $|M|=\mu$, and a set $Q\in\cQ(M)$.
If the event $\cE(M,Q)$ occurs, then there exist $I,g$ such that the event $\cE(M,Q,I,g)$ as in \Lem~\ref{Lemma_complicated} occurs.
Indeed, this is precisely what we pointed out in {\bf A1}, {\bf A2} above.
Thus, by the union bound 
	\begin{eqnarray}\nonumber
	\pr\brk{\cE(M,Q)}&\leq&\sum_{I,g}\pr\brk{\cE(M,Q,I,g)}
		\leq\sum_{\nu=1}^{(k_2-1)\mu}\sum_{I\subset\brk m:|I|=\nu}\sum_{g:Q\ra I\times\brk k}\pr\brk{\cE(M,Q,I,g)}\nonumber\\
		&\leq&\sum_{\nu=1}^{(k_2-1)\mu}\bink{m}{\nu}(k\nu)^{(k_2-1)\mu}\max_{I,g:|I|=\nu,g:Q\ra I\times\brk k}\pr\brk{\cE(M,Q,I,g)}.
			\label{eqactive3}
	\end{eqnarray}
According to \Lem~\ref{Lemma_complicated},
	\begin{eqnarray}\label{eqactive10}
	\pr\brk{\cE(M,Q,I,g)}&\leq&2\bink{T^*+\eps n}{\eps n}(2n)^{-|Q|}2^{-|I\times\brk k|}\exp(1.011k\theta\nu).
	\end{eqnarray}
Combining~(\ref{eqactive3}) and~(\ref{eqactive10}), we obtain
	\begin{eqnarray*}
	\pr\brk{\cE(M,Q)}&\leq&2\bink{T^*+\eps n}{\eps n}(2n)^{-|Q|}
			\sum_{\nu=1}^{(k_2-1)\mu}\bink{m}\nu(k\nu)^{(k_2-1)\mu}2^{-k\nu}\exp(1.011k\theta\nu)\\
		&\leq&2\bink{T^*+\eps n}{\eps n}(2n)^{-|Q|}\sum_{\nu=1}^{(k_2-1)\mu}\bcfr{\eul m}{\nu2^k}^\nu(k\nu)^{(k_2-1)\mu}
			\exp(1.011k\theta\nu)\\
		&\leq&2\bink{T^*+\eps n}{\eps n}(2n)^{-|Q|}\sum_{\nu=1}^{(k_2-1)\mu}\bcfr{\eul \rho n}{k\nu}^\nu(k\nu)^{(k_2-1)\mu}\exp(1.011k\theta\nu).
	\end{eqnarray*}
Since the largest summand is the one with $\nu=(k_2-1)\mu$ and as $|Q|=(k_2-1)\mu$, we obtain
	\begin{eqnarray}\label{eqactive11}
	\pr\brk{\cE(M,Q)}&\leq&2k\mu\bink{T^*+\eps n}{\eps n}
			\bcfr{\exp(1+1.011k\theta)\rho}{2}^{(k_2-1)\mu}
	\end{eqnarray}

Let $\xi>0$ be such that $\bink{k}{k_2-1}=(2\xi)^{k_2-1}$ and let $\zeta=\exp(1+1.011k\theta)$.
Plugging~(\ref{eqactive11}) into~(\ref{eqactive2}), we get
	\begin{eqnarray}\label{eqactive12}
	\pr\brk{\cE(M)}&\leq&2k\mu\bink{T^*+\eps n}{\eps n}\bink{k}{k_2-1}^\mu
			\bcfr{\zeta\rho}{2}^{(k_2-1)\mu}
			\leq2k\mu\bink{T^*+\eps n}{\eps n}\bc{\xi\zeta\rho}^{(k_2-1)\mu}.
	\end{eqnarray}
Finally, (\ref{eqactive1}) and~(\ref{eqactive12}) yield
	\begin{eqnarray}\nonumber
	\pr\brk{\cE}&\leq&2k\mu\bink{T^*+\eps n}{\eps n}\bink m\mu
		\bc{\zeta\xi\rho}^{(k_2-1)\mu}
		\leq2k\mu\bink{T^*+\eps n}{\eps n}\brk{\frac{\eul m}\mu\bc{\zeta\xi\rho}^{(k_2-1)}}^\mu\nonumber\\
		&\leq&2k\mu\bink{T^*+\eps n}{\eps n}\brk{\frac{4\eul 2^k\rho}{k\eps}\bc{\zeta\xi\rho}^{(k_2-1)}}^\mu.
			\label{eqactive13}
	\end{eqnarray}
If $\rho\leq\rho_0=1/25$, then 
	\begin{equation}\label{eqactive13a}
	\frac{4\eul 2^k\rho}{k\eps}\bc{\zeta\xi\rho}^{(k_2-1)}<\exp(-k_2/100)
	\end{equation}
		for $k\geq k_0$ large enough.
Hence, (\ref{eqactive13}) and~(\ref{eqactive13a}) yield for $k\geq k_0$ large enough
	\begin{eqnarray*}\nonumber
	\pr\brk{\cE}
		&\leq&2k\mu\bink{T^*+\eps n}{\eps n}\exp(-k_2\mu/100)
				\leq2k\mu\bcfr{\eul(T^*+\eps n)}{\eps n}^{\eps n+1}\exp(-k_2\mu/100)\\
		&\leq&2k\mu\bcfr{\eul\bc{1/k+\eps}}\eps^{\eps n+1}\exp(-k_2\mu/100)\\
		&\leq&\exp\brk{2\eps n-\eps n\ln\eps-k_2\mu/100+o(n)}\\
		&\leq&\exp\brk{n\bc{2\eps-\eps\ln\eps-k_2\eps/400+o(1)}}\quad\qquad\mbox{[by our choice of $\mu$]}\\
		&\leq&\exp\brk{-n k_2\eps/401}=o(1),\qquad\qquad\ \qquad\qquad\mbox{[by our choice of $\eps$, cf.\ (\ref{eqlambdaeps})]},
	\end{eqnarray*}
as desired.
\qed

\end{document}